\documentclass[12pt]{amsart}
\usepackage{amsmath, amssymb, braket}
\usepackage{tikz-cd}
\usepackage{graphics}
\theoremstyle{plain}
\newtheorem{theorem}{Theorem}[section]
\newtheorem{corollary}[theorem]{Corollary}
\newtheorem{proposition}[theorem]{Proposition}
\newtheorem{lemma}[theorem]{Lemma}

\theoremstyle{definition}
\newtheorem{definition}[theorem]{Definition}
\newtheorem{remark}[theorem]{Remark}
\newtheorem{example}[theorem]{Example}
\newtheorem{problem}[theorem]{Problem}

\numberwithin{equation}{section}

\usepackage{graphicx}

\usepackage{braket}

\usepackage{hyperref}
\input xy
\xyoption{all}

\makeatletter
\@namedef{subjclassname@2020}{\textup{2020} Mathematics Subject Classification}
\makeatother

\begin{document}

\title[Octonions as Clifford-like algebras]{Octonions as Clifford-like algebras}

\author[C.M. Depies]{Connor M. Depies$^1$}
\author[J.D.H. Smith]{Jonathan D. H. Smith$^2$}
\author[M.D. Ashburn]{Mitchell D. Ashburn$^3$}
\address{$^{1,2,3}$Department of Mathematics\\
Iowa State University\\
Ames, Iowa 50011-2104, U.S.A.}
\email{$^1$cmdepies@iastate.edu}
\email{$^2$jdhsmith@iastate.edu}
\email{$^3$ashburn@iastate.edu}

\thanks{}

\begin{abstract}
The associative Cayley-Dickson algebras over the field of real numbers are also Clifford algebras. The alternative but nonassociative real Cayley-Dickson algebras, notably the octonions and split octonions, share with Clifford algebras an involutary anti-automorphism and a set of mutually anticommutative generators. On the basis of these similarities, we introduce \emph{Kingdon algebras}: alternative Clifford-like algebras over vector spaces equipped with a symmetric bilinear form. Over three-dimensional vector spaces, our construction quantizes an alternative non-associative analogue of the exterior algebra. The octonions and split octonions, along with other real generalized Cayley-Dickson algebras in Albert's sense, arise as Kingdon algebras. Our construction gives natural characterizations of the octonion and split octonion algebras by a universality property endowing them with a selected superalgebra structure.
\end{abstract}

\maketitle


\section{Introduction}\label{S:intro}

\subsection{Motivation}

To within isomorphism, there are four different real quadratic normed division algebras: $\mathbb{R}$, $\mathbb{C}$, $\mathbb{H}$ and $\mathbb{O}$; and three split real quadratic normed algebras: the split complex numbers\footnote{The \emph{paracomplex numbers} of  Cruceanu \textit{et al.} \cite{CruForGa}; the \emph{Lorenz numbers} of \cite[Def'n.~6.31]{Harvey}.} $\mathbb R\oplus\mathbb R$, the split quaternions\footnote{The algebra $\mathbb R_2^2$ of real $2\times 2$-matrices \eqref{E:ClifCayl}; Kocik's \emph{pseudo-quaternions} or \emph{kwaternions} \cite[\S3A]{Kocik}.} and the split octonions $\widetilde{\mathbb O}$\footnote{We use the notation of \cite[Def'n.~6.31]{Harvey}.}. Starting from $\mathbb R$, all the further algebras, with their involutary anti-isomorphic \emph{conjugations}, are produced by the iterative \emph{doubling} or \emph{Cayley-Dickson process}
\begin{equation}\label{E:GenCa-Di}
(x_1,y_1)\cdot(x_2,y_2)=
(x_1x_2+\gamma y_2y_1^*,x_1^*y_2+x_2y_1)\,,\
(x,y)^*=(x^*,-y)
\end{equation}
with $r=r^*$ for $r\in\mathbb R$, and $\gamma\in\mathbb R$ \cite[(3)]{Schafer43}, \cite[\S2.2]{SS82}, Section~\ref{SS:Cayl-Dic} below. The associative algebras listed are \emph{Clifford algebras}  \cite[pps. 37-38]{CC1954}, \cite{Chevalley2}, Section~\ref{SS:CliffAlg} below. As for the octonions $\mathbb O$ and split octonions $\widetilde{\mathbb O}$:
\begin{itemize}
\item
Both involve a generating triple
\begin{equation}\label{E:GenACijk}
\set{i,j,k}
\end{equation}
of $\mathbb R$-algebraically independent elements which
\begin{itemize}
\item
individually square to a real number,
\item
mutually anti-commute, and
\item
collectively anti-associate.
\end{itemize}
\item
Both have an involutary anti-isomorphism | their conjugation.
\item
Both admit non-trivial superalgebra structures, i.e., even/odd grading.
\end{itemize}
In short, except for their lack of associativity, these algebras share most of the properties of a Clifford algebra. Our aim in this paper is to present a new, Clifford-like construction process for the algebras, based on a vector space with a bilinear form, as a universal alternative extension that incorporates the action of the form. In particular, we will exhibit the octonions and split octonions as natural witnesses of universality properties.

Our construction method may be contrasted with the method of \cite{GO2011}, which created both Clifford algebras and certain nonassociative algebras as twisted Boolean group algebras. In that paper, two series of $\mathbb{R}$-algebras $\mathbb{O}_{l,m}$ and $\mathbb{M}_{l,m}$ of dimension $2^{l+m}$ appear, with the octonions as $\mathbb{O}_{0,3}$ and the split octonions as $\mathbb{O}_{1,2}$, $\mathbb{O}_{2,1}$ and $\mathbb{O}_{3,0}$. In general, the twisted Boolean group algebras are not alternative. Furthermore, they only include Clifford algebras of spaces with nondegenerate bilinear forms, and thus exclude the exterior algebras.

We define a \emph{formed space} $(V,B)$ to be a finite-dimensional vector space $V$ equipped with a symmetric bilinear form $B$ (Definition~\ref{D:FormedSp}). The associated Clifford algebra $\mathrm{Cl}(V,B)$ is the quotient of the free associative algebra or \emph{tensor algebra} $TV$ over $V$ by the ideal 
\begin{equation}\label{E:ClTAlgId}
J=\braket{xy+yx-B(x,y)\cdot 1|x,y\in V}
\end{equation}
of $TV$ (compare \cite[pp.37--38]{CC1954}, \cite[p.102]{Chevalley2}, \cite[p.179]{Harvey}). Such quotients will not create the nonassociative octonions. For that purpose, consider the real span $E$ of the generating set \eqref{E:GenACijk}. Since the octonions are alternative, we may replace the free associative algebra $TE$ with the \emph{free alternative algebra} $\mathrm{Alt}[E]$ over $E$ (\S\ref{SS:FreeAltA}), and then take its quotient by an appropriate ideal.

The question then is to determine which ideal of $\mathrm{Alt}[E]$ will yield the (split) octonions as its quotient. Note that if the ideal $J$ of \eqref{E:ClTAlgId} was interpreted in $\mathrm{Alt}[E]$, we would have
\begin{equation}\label{E:JijkJijk}
\dim\mathrm{Span}\set{J+i(jk), J+(ij)k}=2
\end{equation}
for the members $i,j,k$ of \eqref{E:GenACijk}, and then the dimension of $\mathrm{Alt}[E]/J$ would be too large. Therefore, we need a rebracketing relation in order to cut down the dimension in \eqref{E:JijkJijk}. In tensor algebra quotients such as Clifford algebras, defining relations that merely involve pairs of vectors are sufficent, since associativity does the rebracketing for us. If only alternativity is given, we cannot do that.

Thus, in the construction that we present in Definition~\ref{D:Altcliffdef}, we will impose a relation involving products of three vectors to capture the anti-associativity of the elements of the generating set \eqref{E:GenACijk}. Motivated by the fact that $x+x^*$ is a field element whenever $x$ is any element of a (generalized) Cayley-Dickson algebra  \cite[Lemma~2.4]{SS82}, we take  
\begin{equation}\label{E:plindrCD}
(ab)c+{((ab)c)}^*=(ab)c-c(ba)=0
\end{equation}
for elements $a,b,c$ of the real span $E$ of the generating set \eqref{E:GenACijk}. 

In general, given a finite-dimensional vector space $V$ over a field $F$ where $2$ is invertible, we will add the \emph{palindromic relators}
\begin{equation}\label{E:plindrom}
(xy)z-z(yx)
\end{equation}
that involve triples $x,y,z$ of vectors from $V$ to the \emph{Clifford relators}
\begin{equation}\label{E:CliffRel}
xy+yx-B(x,y)\cdot 1
\end{equation}
involving pairs of vectors from $V$ and a symmetric bilinear form $B(x,y)$ on $V$ that are already present in \eqref{E:ClTAlgId}, thereby generating an ideal $I_{B}$ of the free alternative algebra $\mathrm{Alt}[V]$ over the vector space $V$. Then, the quotient $K(V,B)$ by this ideal forms an alternative analogue of a Clifford algebra. Using W.K.~Clifford's middle name, the maiden name of his mother \cite{Rooney}, we will refer to these structures as \emph{Kingdon algebras} (Definition~\ref{D:Altcliffdef}). In \S\ref{S:altextalg}, our construction produces an $8$-dimensional alternative algebra $\Lambda'(\mathbb{R}^3)$ that resembles an exterior algebra
\footnote{
It is possible to construct such ``alternative exterior algebras" using vector spaces of higher dimension, but we only briefly mention these in Definition~\ref{D:gnltexal} and Problem~\ref{P:gnltexal}, without further elaboration of their properties in this paper.
}, taking the bilinear form to vanish everywhere. This algebra can be quantized by varying the bilinear form, to yield other $8$-dimensional Kingdon algebras.

\subsection{Plan of the paper}

Background material that goes beyond the motivating discussion is presented in the second chapter. The relevant aspects of alternative algebras are summarized in \S\ref{SS:FreeAltA}. Then, quadratic algebras are treated in \S\ref{SS:QuadrAlg}, placing especial emphasis on the class of symmetric quadratic algebras. Clifford algebras are defined by their universality property in \S\ref{SS:CliffAlg}, and related to their explicit construction as the quotient of a tensor algebra by the ideal \eqref{E:ClTAlgId}.  The Cayley-Dickson process is exhibited in \S\ref{SS:Cayl-Dic} as an iterative process within the class of symmetric quadratic algebras. The final section of the chapter then relates the real Clifford algebras of dimension at most four to their constructions by the Cayley-Dickson process, as summarized in the diagram \eqref{E:ClifCayl}.

Working with an arbitrary field where $2$ is invertible,\footnote{Extension to the characteristic $2$ case is often possible, but comes with a penalty (compare \cite[App.~B]{HVM}).} Definition~\ref{D:Altcliffdef} presents the Kingdon algebra $K(V,B)$ over a formed space $(V,B)$ as the quotient of the free alternative algebra over $V$ by the ideal \eqref{E:altcliffdef} that is generated by the Clifford relators \eqref{E:CliffRel} and the palindromic relators \eqref{E:plindrom}. The Kingdon algebra $K(V,B)$ supports a $\mathbb Z/_2$-grading (Proposition~\ref{P:KingZ2gr}). Proposition~\ref{P:UnvPrKVB} then shows that $K(V,B)$ satisfies the universality property for the domain of an algebra homomorphism extending a suitable linear map from $V$ into an alternative algebra. The universality property furnishes an involutive automorphism and an involutive anti-automorphism analogous to those on Clifford algebras (Corollary~\ref{C:UnvPrKVB}). Theorem~\ref{T:reductionstep} provides a way to rebracket a product of three vectors in a Kingdon algebra as a substitute for the missing associativity. The third chapter culminates with Theorem~\ref{T:constructingoct}, which identifies octonions and split octonions as Kingdon algebras, and thus as solutions to universality properties. Finally, \S\ref{SS:admstrpl} relates the Kingdon algebra specification of the octonions to their traditional specification in terms of an oriented Fano plane.

In the fourth chapter, the $8$-dimensional alternative exterior algebra $\Lambda'(F^3)$ over a $3$-dimensional vector space (Table~\ref{Tb:altextalg}) is introduced, as a parallel to the standard associative exterior algebra $\Lambda(F^3)$ (\S\ref{S:altextalg}). The alternative exterior algebra is identified as the Kingdon algebra $K(F^3,0)$ (Proposition~\ref{P:LamPrKIin}). Theorem~\ref{T:cliffordtocayley} generalizes Theorem~\ref{T:constructingoct} by showing that each Kingdon algebra $K(V,B)$ over a three-dimensional formed space $(V,B)$ is a Cayley-Dickson algebra. It then transpires (\S\ref{SS:KingNorm}) that simplicity of $K(V,B)$ is equivalent to non-degeneracy of the form $B$. With these results in place, Table~\ref{Tb:Mlt8-dim} presents the multiplication of each such Kingdon algebra $K(V,B)$ as a deformation or quantization of the alternative exterior algebra.

The final chapter is devoted to the Kingdon algebras $K(V,B)$ over three-dimensional real formed spaces $(V,B)$, aiming for a classification of these algebras and a determination of their structural properties; most notably their commutants, nuclei, and centers (\S\ref{SS:CeNucCom}). The full classification into isomorphism types is given in \S\ref{SS:FC}, relying on the structural properties to  separate the distinct types. The results are summarized in Table~\ref{Tb:ClassTab}. In \S\ref{S:cayley}, the isomorphism types of Kingdon algebra are located in a graph \eqref{E:KingCayl} of Cayley-Dickson algebras. This graph expands the diagram \eqref{E:ClifCayl} that identified which Cayley-Dickson algebras are Clifford algebras.

\subsection{Notation}

We will often use ``(split) octonions'' as a standard formulation to include both the octonions and the split octonions. In most cases, we will follow the algebraic or diagrammatic convention placing functions after their arguments, either on the line (cf. $n!$) or as a superscript (cf. $x^2$).

\section{Background}\label{S:backgrnd}

\subsection{Alternative algebras}\label{SS:FreeAltA}

Let $A$ be a (not necessarily associative) algebra over a field. Define the \emph{left multiplication}
$$
L(a)\colon A\to A;x\mapsto ax
$$
and
\emph{right multiplication}
$$
R(a)\colon A\to A;x\mapsto xa 
$$
for each element $a$ of $A$. Since $A$ is an algebra, both $L(a)$ and $R(a)$ are endomorphisms of the underlying vector space $A$. In other words, both maps $R\colon A\to\mathrm{End}\,A$ and $L\colon A\to\mathrm{End}\,A$ are linear. Define the \emph{associator}
\begin{align}\notag
(x,y,z)&=(xy)z-x(yz)\\ 
\notag
&
=x[R(y)R(z)-R(yz)] \\ 
\notag
&
=y[L(x)R(z)-R(z)L(x)] \\ 
\notag
&
=z[L(xy)-L(y)L(x)]\,, 
\end{align}
which vanishes if $A$ is associative. The associator is a trilinear function on $A$.

\begin{lemma}\label{L:alternative}
On an algebra $A$ over a field where $2$ is invertible, the following conditions are equivalent: 
\begin{itemize}
\item[$(\mathrm a)$]
The equations
$$
L(a)^2=L(a^2)
\
\mbox{ and }
\
R(a^2)=R(a)^2
$$
hold for all elements $a$ of $A$;
\item[$(\mathrm b)$]
Both the \emph{left alternative identity}
$$
x(xy)=(xx)y
$$
and \emph{right alternative identity}
$$
y(xx)=(yx)x
$$
hold in $A$.
\item[$(\mathrm c)$]
The equations
$$
(x,x,y)=0=(y,x,x)
$$
hold in $A$.
\item[$(\mathrm d)$]
For each element $\pi$ of the symmetric group $S_3$, the equation
\begin{equation}\label{E:AssocAlt}
(x_{1\pi},x_{2\pi},x_{3\pi})=(\mathrm{sgn}\,\pi)(x_1,x_2,x_3)
\end{equation}
holds.
\end{itemize}
\end{lemma}

\begin{proof}
The implications $(\mathrm a)\ \Leftrightarrow\ (\mathrm b)\ \Leftrightarrow\ (\mathrm c)$ are immediate (over any field). If (d) holds and $2$ is invertible, the first equation of (c) follows on setting $x_1=x_2=x$, $x_3=y$, and $\pi=(1\ 2)$, while the second follows on setting $x_1=y$, $x_2=x_3=x$, and $\pi=(2\ 3)$.

Conversly, suppose (c) holds. No assumption is made about the underlying field. Polarization of $(x,x,y)=0$ on the first two arguments shows that $\pi=(1\ 2)$ acts appropriately for (d):
\begin{align*}
0=(x_1+x_2,x_1+x_2,x_3)=(x_1,x_2,x_3)+(x_2,x_1,x_3)\,,
\end{align*}
while polarization of $(y,x,x)=0$ on the second two arguments shows that $\pi=(2\ 3)$ acts appropriately for (d). Since $S_3$ is generated by $\set{(1\ 2), (2\ 3)}$, this suffices for (d).
\end{proof}

\begin{definition}\label{D:alternative}
Consider an algebra $A$.
\begin{enumerate}
\item[$(\mathrm a)$]
The algebra is \emph{alternative} if it satisfies the conditions (a)--(d) of Lemma~\ref{L:alternative}.
\item[$(\mathrm b)$]
The algebra $A$ is \emph{diassociative} if $(x,y,z)=0$ for elements $x,y,z$ of $A$ that lie in a subalgebra $B$ which is generated by a pair of elements.
\item[$(\mathrm c)$]
The algebra is \emph{flexible} if $(xy)x=x(yx)$ for elements $x,y$ of $A$.
\end{enumerate}
\end{definition}

The following is known as Artin's Theorem \cite[Th.~3.1]{S95}, \cite[Th.~2.2]{SS82}.

\begin{theorem}\label{L:2assoc}
Alternative algebras are diassociative.
\end{theorem}

\begin{corollary}\label{C:AltFlexi}
Alternative algebras are flexible.
\end{corollary}

Artin's Theorem is not required for the proof of Corollary~\ref{C:AltFlexi}: It suffices to deduce $(x,y,x)=0$ from \eqref{E:AssocAlt} with $\pi=(1\ 3)$. 
An additional consequence of alternativity is as follows.

\begin{lemma}\label{L:Moufang}\cite[(3.6)]{S95}\cite[Lemma~2.7]{SS82}
In an alternative algebra, the \emph{Moufang} identities
\begin{align}\label{E:Moufang1}
(x(yz))x&=x((yz)x)=(xy)(zx)
\end{align}
are satisfied.
\end{lemma}

Now consider the free unital magma $X^\dag$ over a set $X$. For a field $F$, take the magma algebra $FX^\dag$ over $F$. Note that it contains the $F$-span $FX$ of $X$ as a subset, and a copy $F=F1$ of the field. The algebra $FX^\dag$ may be interpreted as the algebra of ``non-commutative and non-associative polynomials'' over $X$ with coefficients in $F$.

Consider the elements
$$
\lambda(x,y)=(xx)y-x(xy)
\
\mbox{ and }
\
\rho(x,y)=y(xx)-(yx)x
$$
of $F\set{x,y}^\dag$. Note that an $F$-algebra $A$ is alternative if and only if $\lambda(x,y)=0$ and $\rho(x,y)=0$ for all $x,y\in A$. More generally, consider the ideal $\mathcal I(A)$ generated by the set
$$
\set{\lambda(x,y),\rho(x,y)|x,y\in A}
$$
in any $F$-algebra $A$. Then the quotient $A/\mathcal I(A)$ is alternative.

\begin{definition}\label{D:FreeAltX}
Let $X$ be a set. Then the \emph{free alternative algebra} $\mathrm{Alt}(X)$ over $X$ is defined as the quotient $FX^\dag/\mathcal I(FX^\dag)$, taken with the \emph{insertion}
\begin{equation}\label{E:FreeAltX}
\eta_X\colon X\to\mathrm{Alt}(X); x\mapsto x+\mathcal I(FX^\dag)
\end{equation}
of the \emph{generating set} $X$.
\end{definition}

\begin{theorem}\cite[Th.~1.2]{SS82}\label{T:FreeAltX}\\
$(\mathrm a)$
The insertion $\eta_X$ is injective.
\vskip 2mm
\noindent
$(\mathrm b)$
Suppose that $\varphi\colon X\to A$ is a function from $X$ to (the underlying set of) an alternative $F$-algebra $A$. Then there is a uniquely defined homomorphism 
$\Phi\colon\mathrm{Alt}(X)\to A$  such that the diagram 
$$
\xymatrix{
\mathrm{Alt}(X)
\ar@{-->}[dr]^{\Phi}
\\
X
\ar[r]_-{\varphi}
\ar[u]^{\eta_X}
&
A
}
$$
commutes.
\end{theorem}

The universality property of Theorem~\ref{T:FreeAltX} specifies $\mathrm{Alt}(X)$ up to isomorphism. For a vector space $V$ over $F$, with basis $X$, we may thus write $\mathrm{Alt}[V]$ for $\mathrm{Alt}(X)$, and reformulate the universality property as follows. Note that we are identifying $V$ with its isomorphic image in $\mathrm{Alt}[V]$ under the $F$-linear extension of the injective function $\eta_X$.

\begin{corollary}\label{L:UPFA}
Let $V$ be a vector space. Let $\phi:V\to A$ be a linear map from $V$ into an alternative algebra $A$. Then there exists a unique algebra homomorphism $\Phi\colon\mathrm{Alt}[V]\to A$  such that the diagram 
$$
\xymatrix{
\mathrm{Alt}[V]
\ar@{-->}[dr]^{\Phi}
\\
V
\ar[r]_-{\phi}
\ar@{^(->}[u]
&
A
}
$$
commutes.
\end{corollary}

\begin{proof}
Take $\varphi$ in Theorem~\ref{T:FreeAltX}(b) to be the restriction of $\phi$ to $X$.
\end{proof}

\begin{definition}\label{D:faacomponents}
For each natural number $n$, define the subspace of the vector space $\mathrm{Alt}[V]$ spanned by products of exactly $n$ elements of $V$ to be $\mathrm{Alt}[V]_n$. 
\end{definition}

We declare that an element of $\mathrm{Alt}[V]$ is \emph{homogeneous of degree} $n$ if it is an element of $\mathrm{Alt}[V]_n$. The concepts of degree and homogeneity are well defined in the free alternative algebra, since the ideal $\mathcal I(FX^\dag)$ in Definition~\ref{D:FreeAltX} is generated by homogeneous elements \cite[Lemma~1.1]{SS82}.

\subsection{Quadratic algebras}\label{SS:QuadrAlg}

Consider an algebra $A$ over a field $F$ in which $2$ is invertible, where $F[a]$ is commutative for each element $a$ of $A$. The algebra is said to be \emph{symmetric quadratic} if there is an involutary anti-automorphism
\begin{equation}\label{E:ConjInQA}
A\to A;a\mapsto a^*
\end{equation}
known as \emph{conjugation}, restricting for each $a\in A$ to an automorphism $F[a]\to F[a];a\mapsto a^*$, such that
\begin{equation}\label{E:minpolya}
(X-a)(X-a^*)=X^2-T(a)X+N(a)\in F[X]
\end{equation}
with $T(1)=2$ and $N(1)=1$. Note that the linear factors on the left hand side of \eqref{E:minpolya} are elements of the polynomial ring over the commutative ring $F[a]$. Equating coefficients in \eqref{E:minpolya} yields
\begin{equation}\label{E:TrcNormF}
T(a)=a+a^*\in F
\
\mbox{ and }
\
N(a)=aa^*\in F
\end{equation}
for each element $a$ of $A$. The maps $T\colon A\to F$ and $N\colon A\to F$ are respectively known as the \emph{trace} and \emph{norm} of the symmetric quadratic algebra $A$.

\begin{example}\label{X:TrivQuad}
The field $F$ itself is symmetric quadratic, with $a^*=a$ for $a\in F$. It forms the \emph{trivial} symmetric quadratic algebra, with $T(a)=2a$ and $N(a)=a^2$ for $a\in F$.
\end{example}

\begin{example}\label{X:ComplexQ}
The $\mathbb R$-algebra $\mathbb C$ is symmetric quadratic, with \eqref{E:ConjInQA} as the usual conjugation of complex numbers. The trace of a complex number $z$ is twice its real part; its norm is the square $|z|^2$ of its modulus.
\end{example}

\begin{example}\label{X:LorenzQA}
The $F$-algebra $F^2$ is symmetric quadratic, with the conjugation $(a,b)\mapsto(a,-b)$. For the trace, we have $T((a,b))=2a$. For the norm, we have $N((a,b))=a^2-b^2$. 
\end{example}

We now relate the class of symmetric quadratic algebras to a more general class.

\begin{definition}\cite[\S1]{JMO}\label{D:JMOquads}
An $F$-algebra is \emph{quadratic} if it is unital, and for each element $a$ of $A$, the set $\set{1,a,a^2}$ is linearly dependent.
\end{definition}

The following structure theorem for quadratic algebras, which is due to J.M.~Osborn, is motivated by the definition
\begin{equation}\label{E13QuPrd}
(t,\mathbf x)(t',\mathbf x')
=(tt'-\mathbf x\cdot\mathbf x',t\mathbf x'+t'\mathbf x+\mathbf x\times\mathbf x')
\end{equation}
of the quaternion product on the set $\mathbb R\oplus\mathbb R^3$, using the usual symmetric scalar product $\mathbf x\cdot\mathbf x'$ and anticommutative vector product $\mathbf x\times\mathbf x'$ of vectors $\mathbf x,\mathbf x'$ in $\mathbb R^3$ (compare \cite[(2)]{QAA}, \cite[(9.33)]{I2A22}).

\begin{theorem}\cite[Th.~1]{JMO}\label{T:JMOquads}
An $F$-algebra $A$ (of dimension $1+n$) is quadratic if and only if it is isomorphic to the algebra defined on the $F$-space $F\oplus E$ by \eqref{E13QuPrd} for \emph{scalars} $t,t'\in F$ and \emph{vectors} $\mathbf x,\mathbf x'$ in an $n$-dimensional $F$-space $E$ equipped with a bilinear but not necessarily symmetric bilinear form $\mathbf x\cdot\mathbf x'$ and an anticommutative nonunital algebra product $\mathbf x\times\mathbf x'$.
\end{theorem}

We will now specialize Osborn's Theorem to a structure theorem for symmetric quadratic algebras, using a lemma noted by Osborn.

\begin{lemma}\cite[p.203]{JMO}\label{L:SymQuSym}
In the context of Theorem~\ref{T:JMOquads}, the map
\begin{equation}\label{E:InvFplsE}
F\oplus E\to F\oplus E\colon(t,\mathbf x)\mapsto(t,-\mathbf x)
\end{equation}
is an involution if and only if the scalar product $\mathbf x\cdot\mathbf x'$ is symmetric. In particular, these equivalent conditions hold when $A$ is flexible.
\end{lemma}

\begin{theorem}\label{T:SymQuSym}
In the context of Theorem~\ref{T:JMOquads}, the $F$-algebra $A$ is symmetric quadratic if and only if the scalar product $\mathbf x\cdot\mathbf x'$ is symmetric.
\end{theorem}

\begin{proof}
The quadratic algebra $F\oplus E$ defined by a symmetric scalar product is readily confirmed to be a symmetric quadratic algebra with \eqref{E:InvFplsE} as its conjugation. The computations formally track the proof that the real quaternion algebra, implemented with the product \eqref{E13QuPrd} on $\mathbb R\oplus\mathbb R^3$, is a symmetric quadratic algebra.

Conversely, the relation \eqref{E:minpolya} immediately implies that a symmetric quadratic algebra $A$ is quadratic according to Definition~\ref{D:JMOquads}. Applying Theorem~\ref{T:JMOquads} to the quadratic algebra $A$, the involutory conjugation $*\colon A\to A$ transfers to the map \eqref{E:InvFplsE} of $F\oplus E$. Thus, this map is also an involution. Lemma~\ref{L:SymQuSym} then shows that $\mathbf x\cdot\mathbf x'$ is symmetric.
\end{proof}

\subsection{Clifford algebras}\label{SS:CliffAlg}

For our brief summary of Clifford algebras, we will assume that $F$ is a field in which $2$ is invertible.

\begin{definition}\label{D:FormedSp}
A finite-dimensional vector space $V$ over $F$ is said to be a \emph{formed space} $(V,B)$ if it is equipped with a symmetric, bilinear form $B$.
\end{definition}

Since $2$ is invertible, the bilinear form $B(x,y)$ on a formed space $(V,B)$ is equivalent to the quadratic form $Q(x)=\frac12B(x,x)$, using the polarization $B(x,y)=$
\begin{align*}
&\tfrac12[B(x,x)+B(x,y)+B(y,x)+B(y,y)]-\tfrac12B(x,x)-\tfrac12B(y,y)
\\
&
=\tfrac12B(x+y,x+y)-\tfrac12B(x,x)-\tfrac12B(y,y)
=Q(x+y)-Q(x)-Q(y)
\end{align*}
to recover the bilinear form from the quadratic form. We define Clifford algebras as solutions to a universality property.

\begin{definition}
Suppose that $(V,B)$ is a formed space. Then the \emph{Clifford algebra} $\mathrm{Cl}(V,B)$ satisfies the following universality property. Suppose that $A$ is an associative $F$-algebra, and that $\theta\colon V\to A;v\mapsto v'$ is a linear map such that
 $(u')^2=Q(u)\cdot 1$ or
\begin{align}\label{E:MatchCon}
u'v'+v'u'&=B(u,v)\cdot 1
\end{align}
for all $u,v\in V$ --- compare \eqref{E:CliffRel}. Then there exists a unique $F$-algebra homomorphism $\Theta\colon \mathrm{Cl}(V,B)\to A$ such that the diagram
$$
\xymatrix{
\mathrm{Cl}(V,B)
\ar@{-->}[dr]^{\Theta}
\\
V
\ar[u]^{\iota}
\ar[r]_{\theta}
&
A
}
$$
commutes.
\end{definition}

\begin{remark}
There are a number of different possible conventions for the relation of the quadratic or bilinear forms to the multiplication in the Clifford algebra. In particular, some are designed to handle fields of characteristic $2$. Our conventions follow \cite[p.103 (1),(2)]{Chevalley2}. To obtain $\mathbb C$ as a real Clifford algebra, our convention is chosen to produce $i^2=Q(i)=-1$, corresponding to $B(i,i)=ii+ii=-2$. This convention simplifies relations such as \eqref{E:MatchCon}. Moreover, when extended to Clifford-like alternative algebras over general fields where $2$ is invertible, it produces simple products like those exhibited in Table~\ref{Tb:Mlt8-dim} below, and facilitates the statement and proof of Theorem~\ref{T:reductionstep}. 
\end{remark}

Explicitly, as was discussed in the introduction, the Clifford algebra $\mathrm{Cl}(V,B)$ may be constructed as the quotient of the tensor algebra $TV$ over $V$ by the ideal \eqref{E:ClTAlgId}of $TV$. 
When the formed space $(V,B)$ has dimension $n$, then the Clifford algebra $\mathrm{Cl}(V,B)$ has dimension $2^n$, and includes (a faithful copy of) the vector space $V$ \cite[p.104]{Chevalley2}. The examples of real Clifford algebras that are relevant to this paper are displayed in the diagram \eqref{E:ClifCayl} below, using the notation \eqref{E:Cb0b1b-1}.

The tensor algebra $TV$ is graded by $\mathbb N$. For each natural number $n$, its homogeneous component $TV_n$ consists of linear combinations of products of $n$ vectors from $V$. However, since the Clifford relators \eqref{E:CliffRel} are not homogeneous, the quotient $\mathrm{Cl}(V,B)$ is not $\mathbb N$-graded. Even so, the Clifford relators are linear combinations of evenly graded elements, so the quotient is $\mathbb Z/_2$-graded, becoming a \emph{superalgebra}. The (copy of the) field $F$ is even, while (the copy of) $V$ is odd.

\subsection{The Cayley-Dickson process}\label{SS:Cayl-Dic}

We consider the Cayley-Dickson process, in  Albert's broad sense with a parameter $\gamma$ as below \cite[(58)]{AAA42}, \cite[\S3.1]{DSVM}, \cite[(3)]{Schafer43}, \cite[(1)]{Schafer54}, \cite[\S2.2]{SS82}. Noting the detailed differences between the various presentations of the process, we are following the formulation of \cite[\S3.1]{DSVM}, \cite[(3)]{Schafer43} and \cite[\S2.2]{SS82} in this paper, as displayed in \eqref{E:GenCa-Di}. We again take $F$ to be a field where $2$ is invertible.

\begin{definition}\label{D:Cayley}
Suppose that $A$ is a symmetric quadratic $F$-algebra. Define the multiplication
$
(a,b)\cdot(c,d)=(ac+\gamma d^*b, da+bc^*)
$
on $A^2$, where $\gamma$ is an element of $F$. Write $A^\gamma$ for the resulting algebra structure on the vector space $A^2$. This construction of $A^\gamma$ from $A$ is the \emph{Cayley-Dickson process} (as presented by \eqref{E:GenCa-Di} in the real case), and the resulting algebra $A^\gamma$ is described as a \emph{Cayley-Dickson algebra}.
\end{definition}

The following proposition lists key properties of the Cayley-Dickson process. Together with Definition~\ref{D:Cayley}, Proposition~\ref{P:ropofC-D}(a) ensures the iterative nature of the process.

\begin{proposition}\label{P:ropofC-D}
Consider the context of Definition~\ref{D:Cayley}.
\vskip 2mm
\noindent
$(\mathrm a)$
With conjugation
$
*\colon A^\gamma\to A^\gamma;(a,b)\mapsto(a^*,-b)
$
and
\begin{equation}\label{E:TrNrGamm}
T((a,b))=T(a)\in F\,,
\quad
\
N((a,b))=N(a)-\gamma N(b)\in F\,,
\end{equation}
the algebra $A^\gamma$ is a symmetric quadratic $F$-algebra.
\vskip 2mm
\noindent
$(\mathrm b)$
The algebra $A^\gamma$ is commutative if and only if $A$ is commutative and its conjugation is trivial: $*\colon A\to A;a\mapsto a$.
\vskip 2mm
\noindent
$(\mathrm c)$
The algebra $A^\gamma$ is associative if and only if $A$ is associative and commutative.
\vskip 2mm
\noindent
$(\mathrm d)$
The algebra $A^\gamma$ is alternative if and only if $A$ is associative.
\end{proposition}

The following lemma shows that when the base field of $A$ is $\mathbb{R}$,  it is sufficient to limit ourselves to the three cases where $\gamma\in\set{0,\pm1}$.

\begin{lemma}
Consider an $\mathbb R$-algebra $A$. If $0\ne\gamma\in\mathbb R$, then
$$
A^{\gamma/|\gamma|}\to A^\gamma;(1,0)\mapsto(1,0)\,,\ (0,1)\mapsto(0,|\gamma|^{-1/2})
$$
is an isomorphism.
\end{lemma}

\subsection{Clifford algebras and the Cayley-Dickson process}

All the $1$-, $2$-, and $4$-dimensional real Clifford algebras may be construed as real Cayley-Dickson algebras, as displayed in the following diagram:
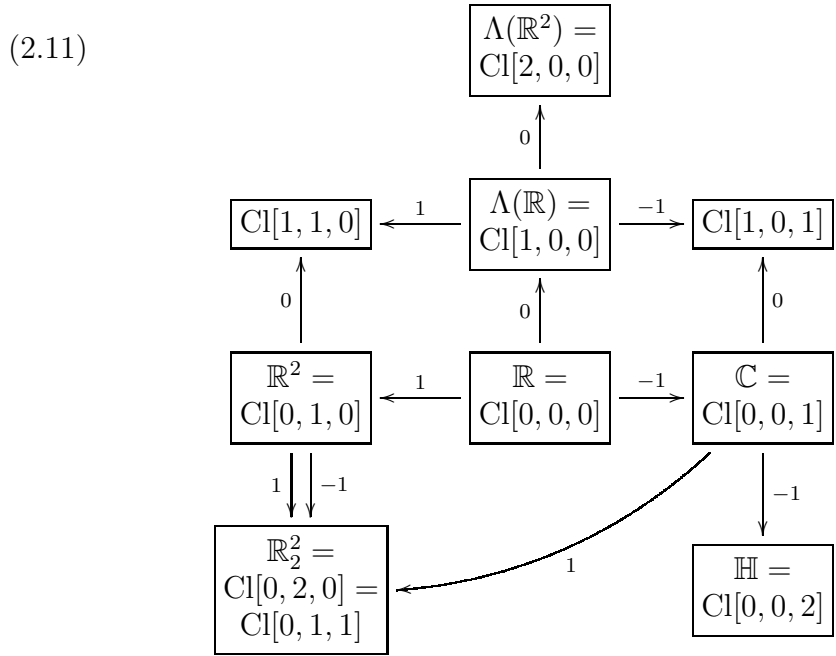
\begin{equation}\label{E:ClifCayl}
\xymatrix{
&
\protect{
\boxed{
\begin{matrix}
\Lambda(\mathbb R^2)=\\
\mathrm{Cl}[2,0,0]
\end{matrix}
}
}
\\
\boxed{
\mathrm{Cl}[1,1,0]
}
&
\protect{
\boxed{
\begin{matrix}
\Lambda(\mathbb R)=\\
\mathrm{Cl}[1,0,0]
\end{matrix}
}
}
\ar[r]^-{-1}
\ar[l]_-{1}
\ar[u]^-{0}
&
\boxed{
\mathrm{Cl}[1,0,1]
}
\\
\protect{
\boxed{
\begin{matrix}
\mathbb R^2=\\
\mathrm{Cl}[0,1,0]
\end{matrix}
}
}
\ar[u]^-{0}
\ar@<.7ex>[d]^-{-1}
\ar@<-.7ex>[d]_-{1}
&
\protect{
\boxed{
\begin{matrix}
\mathbb R=\\
\mathrm{Cl}[0,0,0]
\end{matrix}
}
}
\ar[r]^-{-1}
\ar[l]_-{1}
\ar[u]^-{0}
&
\protect{
\boxed{
\begin{matrix}
\mathbb C=\\
\mathrm{Cl}[0,0,1]
\end{matrix}
}
}
\ar[u]_-{0}
\ar[d]^-{-1}
\ar@/^2pc/[dll]^-{1}
\\
\protect{
\boxed{
\begin{matrix}
\mathbb R_2^2=\\
\mathrm{Cl}[0,2,0]
=\\
\mathrm{Cl}[0,1,1]
\end{matrix}
}
}
&
&
\protect{
\boxed{
\begin{matrix}
\mathbb H=\\
\mathrm{Cl}[0,0,2]
\end{matrix}
}
}
}
\end{equation}
Here, each arrow
$
\xymatrix{
A
\ar[r]^{\gamma}
&
A^\gamma
}
$
records an instance of the Cayley-Dickson process. The notation 
\begin{equation}\label{E:Cb0b1b-1} 
\mathrm{Cl}[b_0,b_1,b_{-1}]\,,
\end{equation}
with natural numbers  $b_0,b_1,b_{-1}$, denotes the Clifford algebra of the real formed space $(V,B)$ of dimension $b_0+b_1+b_{-1}$ having an orthogonal basis with $b_l$ elements $e$ satisfying $B(e,e)=2l$, for each ``length'' $l$ in $\set{0,\pm1}$.  The algebra $\mathbb R_2^2$ is the Clifford algebra $\mathrm{Cl}[0,2,0]\cong\mathrm{Cl}[0,1,1]$ of real $2\times 2$-matrices. The Clifford algebras $\mathrm{Cl}[1,0,0]$ and $\mathrm{Cl}[2,0,0]$, over $\mathbb R$ and $\mathbb R^2$ respectively with trivial bilinear forms, are exterior algebras.

\section{Kingdon algebras}\label{S:const}

Throughout this section, we will take $F$ as a field in which $2$ is invertible.

\subsection{Kingdon algebras of formed spaces}

Consider a formed space $(V,B)$ over $F$, as in Definition~\ref{D:FormedSp}. For the following, interpret $\mathrm{Alt}[V]$ as the free alternative $F$-algebra over $V$, recalling that it includes the field $F$ and the vector space $V$ as the respective subspaces $\mathrm{Alt}[V]_0$ and $\mathrm{Alt}[V]_1$ of Definition~\ref{D:faacomponents}.

\begin{definition}\label{D:Altcliffdef}
Let $(V,B)$ be a formed space.
\begin{enumerate}
\item[$(\mathrm a)$]
Consider the subset
$$
\set{uv+vu-B(u,v),(uv)w-w(vu)|u,v,w\in V}
$$
of the free alternative algebra $\mathrm{Alt}[V]$ over $V$. Define $I_{B}$ as the ideal
\begin{equation}\label{E:altcliffdef}
\braket{uv+vu-B(u,v),(uv)w-w(vu)|u,v,w\in V}
\end{equation}
of $\mathrm{Alt}[V]$ generated by the subset.
\item[$(\mathrm b)$]
The quotient $\mathrm{Alt}[V]/I_{B}$ is the \emph{Kingdon algebra} $K(V,B)$ over $(V,B)$.
\item[$(\mathrm c)$]
The linear map
\begin{equation}\label{E:InsV2KBT}
\iota\colon V\to K(V,B);v\mapsto I_{B}+v
\end{equation}
is called the \emph{insertion}.
\end{enumerate}
\end{definition}

\begin{proposition}\label{P:KingZ2gr}
The Kingdon algebra $K(V,B)$ is graded by $\mathbb Z/_2$.
\end{proposition}

\begin{proof}
The free alternative algebra $\mathrm{Alt}\,V$ is $\mathbb Z$-graded \cite[\S1.3]{SS82}, and therefore also $\mathbb Z/_2$-graded. Its ideal $I_B$ is $\mathbb Z/_2$-homogeneous. Thus, the quotient $\mathrm{Alt}\,V/I_B$ is $\mathbb Z/_2$-graded \cite[\S I.2]{Chevalley1}, \cite[pp.9--10]{Chevalley2}.
\end{proof}

The following result exhibits the universality property satisfied by the Kingdon algebra $K(V,B)$ over a formed space $(V,B)$.

\begin{proposition}\label{P:UnvPrKVB}
Suppose that $A$ is an alternative algebra. Suppose that $\theta\colon V\to A;v\mapsto v'$ is a linear map such that
\begin{align}\label{E:thetaVAB}
u'v'+v'u'&=B(u,v)\cdot 1\quad \mbox{ and }
\\ \label{E:thetaVAT}
(u'v')w'&=w'(v'u')
\end{align}
for all elements $u,v,w$ of $V$. Then there exists a unique $F$-algebra homomorphism $\Theta\colon K(V,B)\to A$ such that the diagram
$$
\xymatrix{
K(V,B)
\ar@{-->}[dr]^{\Theta}
\\
V
\ar[u]^{\iota}
\ar[r]_{\theta}
&
A
}
$$
commutes.
\end{proposition}

\begin{proof}
Consider the subspace $\set{(I_{B}+v,v')|v\in V}$ of $K(V,B)\oplus A$. Then the subalgebra of $K(V,B)\oplus A$ that it generates is the graph of the $F$-homomorphism $\Theta$.
\end{proof}

\begin{corollary}\label{C:UnvPrKVB}
Consider  the negation map 
\begin{equation}\label{E:NegatonV}
\nu\colon
V^\iota\to V^\iota;
I_{B}+v\mapsto I_{B}-v
\end{equation}
on the image of the insertion \eqref{E:InsV2KBT}.
\begin{enumerate}
\item[$(\mathrm a)$]
The Kingdon algebra $K(V,B)$ of $(V,B)$ has a uniquely defined involutory automorphism $\nu$ extending \eqref{E:NegatonV}.
\item[$(\mathrm b)$]
The Kingdon algebra $K(V,B)$ of $(V,B)$ has a uniquely defined involutory anti-automorphism $\kappa$ extending \eqref{E:NegatonV}.
\end{enumerate}
\end{corollary}

\begin{proof}
In each case, we take a linear map
\begin{equation}\label{E:thetViKo}
\theta\colon V^\iota\to K(V,B);I_{B}+v\mapsto I_{B}-v
\end{equation}
which corestricts to $\nu$, and invoke the universality property to obtain a linear extension $\Theta\colon K(V,B)\to K(V,B)$. To see that $\Theta^2$ is the identity automorphism on $K(V,B)$, note that \eqref{E:NegatonV} is involutory. Therefore, the extension $\Theta^2$ of its square
is the identity map on $K(V,B)$. It then follows that $\Theta$ is invertible. In the remainder of the proof, we just use a representative vector $x$ from $V$ as an abbreviation for its coset $I_B+x$ in the image of the insertion $\iota\colon V\to K(V,B)$.
\vskip 2mm
\noindent
(a)
Take $K(V,B)$ as the target alternative algebra $A$ in Proposition~\ref{P:UnvPrKVB}. To verify \eqref{E:thetaVAB}, note
\begin{align*}
(-u)(-v)+(-v)(-u)
=uv+vu
=B(u,v)
\end{align*}
for all $u,v$ in $V$. To verify \eqref{E:thetaVAT}, note
\begin{align*}
\Big\{&\big[(-u)(-v)\big](-w)\Big\}-\Big\{(-w)\big[(-v)(-u)\big]\Big\}
=w(vu)-(uv)w
=0
\end{align*}
for all $u,v,w$ in $V$. The homomorphism $\nu\colon K(V,B)\to K(V,B)$ is the lifting $\Theta$ of \eqref{E:thetViKo} furnished by the universality property.
\vskip 2mm
\noindent
(b)
Take the opposite $K(V,B)^\mathsf{op}$ of the Kingdon algebra $K(V,B)$ as the target alternative algebra $A$ in Proposition~\ref{P:UnvPrKVB}. Interpret \eqref{E:thetViKo} as a linear map
$
\theta\colon V^\iota\to K(V,B)^\mathsf{op}
$.
Write $\circ$ for the product in $K(V,B)^\mathsf{op}$. To verify \eqref{E:thetaVAB} recalling the symmetry of $B$, note the computation
\begin{align*}
(-u)\circ(-v)+(-v)\circ(-u)
=vu+uv
=B(v,u)=B(u,v)
\end{align*}
for all $u,v$ in $V$. To verify \eqref{E:thetaVAT}, note
\begin{align*}
\Big\{&\big[(-u)\circ(-v)\big]\circ(-w)\Big\}-\Big\{(-w)\circ\big[(-v)\circ(-u)\big]\Big\}
\\
&
=\big[(vu)\circ(-w)\big]-\big[(-w)\circ(uv)\big]
=(uv)w-w(vu)
=0
\end{align*}
for all $u,v,w$ in $V$. The homomorphism $\kappa\colon K(V,B)\to K(V,B)^\mathsf{op}$ is the lifting $\Theta$ of \eqref{E:thetViKo} furnished by the universality property. 
\end{proof}

\begin{definition}
Consider a Kingdon algebra $K(V,B)$ over a formed space $(V,B)$.
\begin{enumerate}
\item[$(\mathrm a)$]
The involutary automorphism $\nu\colon K(V,B)\to K(V,B)$ specified in Corollary~\ref{C:UnvPrKVB}(a) is called the \emph{main automorphism} of $K(V,B)$.
\item[$(\mathrm b)$]
The involutary anti-automorphism $\kappa\colon K(V,B)\to K(V,B)$ of Corollary~\ref{C:UnvPrKVB}(b) is called the \emph{Clifford conjugation} of $K(V,B)$.
\end{enumerate}
\end{definition}

\subsection{Bases of the formed space}

For an $n$-dimensional formed space $(V,B)$, it is often convenient to consider the Kingdon algebra $K(V,B)$ in terms of a basis $\set{e_i|1\le i\le n}$ of $V$. We present two lemmas that are relevant to this approach.

\begin{lemma}\label{L:BsGenKVB}
The Kingdon algebra $K(V,B)$ is generated, as an algebra over the field $F$, by the basis $\set{e_i|1\le i\le n}$.
\end{lemma}

\begin{proof}
Bearing Definition~\ref{D:Altcliffdef}(b) in mind, it suffices to note that the $F$-algebra $\mathrm{Alt}[V]$ is generated by $\set{e_i|1\le i\le n}$.
\end{proof}

The ideal $I_{B}$ of \eqref{E:altcliffdef} may be expressed nicely in terms of the basis of the formed space $(V,B)$.

\begin{lemma}\label{L:standardgenerators}
The ideal
\begin{equation}\label{E:BasIdlJB}
J_{B}=\braket{e_ie_j+e_je_i-B(e_i,e_j),(e_ie_j)e_k-e_k(e_je_i)|1\le i,j,k\le n}
\end{equation} 
of $\mathrm{Alt}[V]$ is equal to the ideal $I_{B}$ of \eqref{E:altcliffdef}.
\end{lemma}

\begin{proof}
Certainly $J_{B}\subseteq I_{B,}$. We prove the reverse inclusion. Consider linear combinations 
$$
u=\sum_{i=1}^n a_ie_i\,,
\quad
v=\sum_{i=1}^n b_ie_i\,,
\quad
w=\sum_{i=1}^n c_ie_i
$$
in $V$. Then $uv+vu-B(u,v)=$
\begin{align*}
\Big(\sum_{i=1}^n a_ie_i\Big)&\Big(\sum_{j=1}^n b_je_j\Big)
+\Big(\sum_{j=1}^n b_je_j\Big)\Big(\sum_{i=1}^n a_ie_i\Big)
-B\Big(\sum_{i=1}^n a_ie_i,\sum_{j=1}^n b_je_j\Big)
\\
&
=
\sum_{1\le i, j\le n}^na_ib_j\big[(e_ie_j+e_je_i)-B(e_i,e_j)\big]\in J_B
\end{align*}
and $(uv)w-w(vu)=$
\begin{align*}
\bigg[\Big(\sum_{i=1}^n a_i&e_i\Big)\Big(\sum_{j=1}^n b_je_j\Big)\bigg]
\Big(\sum_{k=1}^n c_ke_k\Big)
\\
-&
\Big(\sum_{k=1}^n c_ke_k\Big)
\bigg[\Big(\sum_{=1}^n b_je_j\Big)\Big(\sum_{i=1}^n a_ie_i\Big)\bigg]
\\
=&\sum_{1\le i,j,k\le n}a_ib_jc_k \big[(e_ie_j)e_k-e_k(e_je_i)\big]\in J_B
\end{align*}
as required.
\end{proof}

\subsection{The reduction process}

Starting from the generating triple \eqref{E:GenACijk} for the (split) octonions, we have
$
\set{1,i,j,k,ij,jk,ik,(ij)k}
$
as a basis. Within the Kingdon algebra $K(V,B)$ of a formed space $(V,B)$ over a field where $2$ is invertible, we demonstrate how a product of three (images under $\iota$ of) vector elements $a,b,c$ may be reduced to a linear combination of $(ab)c$ and products of fewer vector elements. As before, a vector $a$ in $V$ is used to represent the coset $a+I_B$ in $K(V,B)$.

\begin{theorem}\label{T:reductionstep}
For (images under $\iota$ of) vectors $a,b,c$ from $V$, the identity
\begin{equation}\label{E:res}
(ab)c+a(bc)=aB(b,c)-bB(c,a)+cB(a,b)
\end{equation}
holds in $K(V,B)$.
\end{theorem}

\begin{proof}
Starting from alternativity and $ab+ba=B(a,b)$, we have
\begin{align*}
0=(c,a,b)+(c,b,a)
&=(ca)b-c(ab)+(cb)a-c(ba)
\\
&
=(ca)b-c(ab)+(cb)a+c(ab)-cB(a,b)\,,
\end{align*}
which cancels to
\begin{align}\label{E:7}
(ca)b+(cb)a-cB(a,b)&=0\,,\ \mbox{becoming}
\\ \label{E:8}
(ac)b+(ab)c-aB(c,b)&=0
\end{align}
after interchange of $a$ and $c$. Now $(cb)a=a(bc)$, and thus \eqref{E:7} gives
\begin{equation}\label{E:2}
0=(ca)b+a(bc)-cB(a,b)\,.
\end{equation}
Also, recalling $ca+ac=B(c,a)$, note that \eqref{E:8} yields
\begin{equation}\label{E:4}
0=-(ca)b+bB(c,a)+(ab)c-aB(c,b)\,.
\end{equation}
Adding \eqref{E:4} to \eqref{E:2} yields the required equation \eqref{E:res}.
\end{proof}

\subsection{Low-dimensional examples}\label{SS:LowDimEx}

Consider a formed space $(V,B)$.

\begin{proposition}\label{P:4cliff}
If $\dim(V)<3$, then $K(V,B)$ is the (associative) Clifford algebra $\mathrm{Cl}(V,B)$.
\end{proposition}

\begin{proof}
If $V$ has a basis with less than $3$ elements, Artin's Theorem shows that the alternative algebra $\mathrm{Alt}[V]$ is actually associative, and thereby coincides with the tensor algebra $TV$ over $V$. Thus, the terms $(e_ie_j)e_k-e_k(e_je_i)$ of \eqref{E:BasIdlJB} vanish, and the specification Definition~\ref{D:Altcliffdef}(b) of the Kingdon algebra $K(V,B)$ reduces to the specification of the Clifford algebra $\mathrm{Cl}(V,B)$ as the quotient of the tensor algebra $TV$ by the ideal \eqref{E:ClTAlgId}.
\end{proof}

\subsection{(Split) octonions}\label{SS:soctnons}

We will now show how the (split) octonions emerge as Kingdon algebras. Take a $3$-dimensional vector space $V$ spanned by $\set{i,j,k}$ over a field $F$ in which $2$ is invertible. Suppose that the set $\set{i,j,k}$ is orthogonal with respect to a symmetric bilinear form $B$ on $V$.  Consider the insertion $\iota\colon V\to K(V,B);a\mapsto a+I_B$, again using $a$ as a representative for $a+I_B$.

\begin{lemma}\label{L:soctnons}
The equations
\begin{align}\label{E:FormsOfO}
&
(ij)k=(jk)i=(ki)j=-i(jk)=-j(ki)=-k(ij)
\\ \notag
=\,
&
i(kj)=j(ik)=k(ji)=-(ik)j=-(ji)k=-(kj)i
\end{align}
hold in $K(V,B)$.
\end{lemma}

\begin{proof}
Consider the cyclically ordered set $\set{i<j<k<i}$. The first relations of $J_B$ from \eqref{E:BasIdlJB} and the orthogonality yield the equation $-i(jk)=i(kj)$ and its cyclic permutations, such as 
\begin{equation}\label{E:-kij=kji}
-k(ij)=k(ji)\,.
\end{equation}
Together with the orthogonality, Theorem~\ref{T:reductionstep} yields $(ij)k=-i(jk)$ and its cyclic permutations, along with $i(kj)=-(ik)j$ and its cyclic permutations.
Thus, for each $1\le r\le 3$, the $r$-th, $(r+3)$-rd, $(r+6)$-th and $(r+9)$-th terms of \eqref{E:FormsOfO} agree.

Now, the second relations of $J_B$ from \eqref{E:BasIdlJB} yield $i(kj)=(jk)i$ and $(ij)k=k(ji)\overset{\eqref{E:-kij=kji}}=-k(ij)$, respectively implying the equality of both the $r=1$ and $r=2$ terms, and of the $r=1$ and $r=3$ terms, in the equations \eqref{E:FormsOfO}.
\end{proof}

\begin{definition}\label{D:FormsOfO}
Write $\omega$ (as a \emph{volume element}) for the common value of the terms appearing in \eqref{E:FormsOfO}.
\end{definition}

\begin{lemma}\label{L:SisClosd}
The $F$-span of the subset 
$$
S=\set{\pm1,\pm i,\pm j,\pm k,\pm ij,\pm jk,\pm ki,\pm\omega}
$$
of $K(V,B)$ is closed under multiplication.
\end{lemma}

\begin{proof}
By Proposition~\ref{P:4cliff}, the $F$-spans of the subsets
$$
\set{\pm1,\pm i,\pm j,\pm ij}\,,\
\set{\pm1,\pm j,\pm k,\pm jk}\,,\ \mbox{and }
\set{\pm1,\pm k,\pm i,\pm ki}
$$
are closed under multiplication. By Lemma~\ref{L:soctnons}, the products with domains
$$
\set{\pm ij,\pm jk,\pm ki}
\times
\set{\pm i,\pm j,\pm k}
\mbox{and}
\set{\pm i,\pm j,\pm k}
\times
\set{\pm ij,\pm jk,\pm ki}
$$
lie in $\set{\pm\omega}$.
The products with domains
$$
\set{\pm\omega}
\times
\set{\pm i,\pm j,\pm k}
\mbox{ and }
\set{\pm i,\pm j,\pm k}
\times
\set{\pm\omega}
$$
lie in the $F$-span of $\set{\pm ij,\pm jk,\pm ki}$ since one has
\begin{equation*}
\omega i\overset{\eqref{E:FormsOfO}}=[(jk)i]i\overset{\mathrm{Th.}~\ref{L:2assoc}}=(jk)i^2=\tfrac12jkB(i,i)
\end{equation*}
and
$$
i\omega\overset{\eqref{E:FormsOfO}}=-i[i(jk)]\overset{\mathrm{Th.}~\ref{L:2assoc}}=-i^2(jk)=-\tfrac12jkB(i,i)
$$
together with their images under elements of the cyclic group $\braket{(i\ j\ k)}$. Finally, the computations $(jk)(ij)\overset{\eqref{E:Moufang1}}=[j(ki)]j\overset{\eqref{E:FormsOfO}}=$
\begin{equation}\label{E:OmjikBJJ}
\omega j\overset{\eqref{E:FormsOfO}}=[(ki)j]j\overset{\mathrm{Th.}~\ref{L:2assoc}}=(ki)j^2=\tfrac12kiB(j,j)\,,
\end{equation}
$$
(ij)(jk)=(ji)(kj)\overset{\eqref{E:Moufang1}}=[j(ik)]j\overset{\eqref{E:FormsOfO}}=-\omega j
\overset{\eqref{E:OmjikBJJ}}=-\tfrac12kiB(j,j)\,,
$$
and
$$
(ij)(ij)=-(ij)(ji)\overset{\eqref{E:Moufang1}}=-(i(jj))i=-\tfrac12B(j,j)i^2=-\tfrac14B(i,i)B(j,j)
\,,
$$
together with their images under elements of the cyclic group $\braket{(i\ j\ k)}$,
serve to show that the $F$-span of the set $\set{\pm1,\pm ij,\pm jk,\pm ki}$ is closed under multiplication.
\end{proof}

\begin{corollary}
If $\set{B(a,a)|a\in\set{i,j,k}}\subseteq\set{\pm2}$, then the set $\set{\pm1,\pm ij,\pm jk,\pm ki}$ is closed under multiplication.
\end{corollary}

\begin{theorem}\label{T:constructingoct}
Consider a three-dimensional vector space $V$ spanned by $\set{i,j,k}$ over a field $F$ which is not of characteristic 2. Suppose that the set $\set{i,j,k}$ is orthogonal with respect to a symmetric bilinear form $B$ on $V$.
\begin{enumerate}
\item[$(\mathrm a)$]
Suppose $B(i,i)=B(j,j)=B(k,k)=-2$. Then the Kingdon algebra $K(V,B)$ is isomorphic to the octonion algebra $\mathbb O$ over $F$.
\item[$(\mathrm b)$]
Suppose $B(i,i)=B(j,j)=-2$ and $B(k,k)=2$. Then the Kingdon algebra $K(V,B)$ is isomorphic to the split octonion algebra $\widetilde{\mathbb O}$ over $F$.
\end{enumerate}
\end{theorem}

\begin{proof}
We prove (a). The proof of (b) is similar. By Lemma~\ref{L:SisClosd}, it is apparent that $K(V,B)$ is an at most $8$-dimensional alternative algebra over the field $F$.

We have an injective linear map $\psi:V\to \mathbb{O}$ which takes $i$ to $i$, $j$ to $j$ and $k$ to $k$. Then, by Corollary~\ref{L:UPFA} and the alternativity of the octonions, we have the commutative diagram 
$$
\xymatrix{
\mathrm{Alt}[V]
\ar@{-->}[dr]^{\Psi}
\\
V
\ar[r]_{\psi}
\ar@{^(->}[u]
&
\mathbb O
}
$$
with an algebra homomorphism $\Psi$ extending $\psi$. Since the image of $V$ generates all of $\mathbb{O}$, the homomorphism $\Psi$ is surjective. The kernel of $\Psi$ includes the ideal $J_B$ of $\mathrm{Alt}[V]$ generated by
$$
\set{ab+ba-B(a,b)\cdot 1, (ab)c-c(ba)|a,b,c\in\set{i,j,k}}
$$ 
(cf. Lemma~\ref{L:standardgenerators}). Thus, we have a surjective algebra homomorphism
\begin{equation}\label{E:PrimePsi}
\Psi':K(V,B)=\mathrm{Alt}[V]/J_B\to\mathbb{O}\,.
\end{equation} 
This map is a surjective linear transformation from a space of dimension at most $8$ to one of dimension exactly $8$. Therefore, $K(V,B)$ must also have dimension at least $8$, and hence exactly $8$. We take \eqref{E:PrimePsi} as the claimed isomorphism.
\end{proof}

\begin{remark}\label{R:constructingoct}
(a)
In conjunction with  Proposition~\ref{P:UnvPrKVB}, Theorem~\ref{T:constructingoct} constructs the (split) octonions by a universality property.
\vskip 2mm
\noindent
(b)
Our universality property of the octonions may replace the use of Cayley-Dickson algebra properties in \cite[\S6.8]{ConwaySmith} for the determination of the rank of the real form of the Lie group $G_2$, the automorphism group of the real octonion algebra. An automorphism is fixed by its action on the set $\set{i,j,k}$ of Theorem~\ref{T:constructingoct}. The image of $i$ may first be chosen to lie anywhere on the $6$-dimensional manifold of octonion elements of octonion norm $1$ in the $7$-dimensional orthogonal complement of $\mathbb R$ under the inner product defined by the octonion norm. Next, the image of $j$ may be chosen to lie anywhere on the $5$-dimensional manifold of octonion elements of octonion norm $1$ in the $6$-dimensional orthogonal complement of the subalgebra generated by $\set i$. Finally, the image of $k$ may be chosen to lie anywhere on the $3$-dimensional manifold of octonion elements of octonion norm $1$ in the $4$-dimensional orthogonal complement of the subalgebra generated by $\set{i,j}$. Thus, the rank of $G_2$ is obtained as $6+5+3=14$.
\end{remark}

\subsection{Admissible triples}\label{SS:admstrpl}

In the octonion algebra $\mathbb O$, the triples that appear in the set
\begin{equation}\label{E:admstrpl}
\mathcal A=
\begin{Bmatrix}
&(i,j,ij), &(j,k,jk), &(k,i,ki),\\
\\
&(i,\omega,jk), &(j,\omega,ki), &(k,\omega,ij),\\
\\
& &(ij,jk,ki)
\end{Bmatrix}\,,
\end{equation}
understood along with their cyclic permutations (for example $(j,ij,i)$, $(ij,i,j)$, and so on) are described as \emph{admissible} \cite[p.676]{GHW}. If $(x,y,z)$ is an admissible triple, then $xy=z$ and $yx=-z$ in $\mathbb O$. Note that the final triple displayed in \eqref{E:admstrpl} is the only one that is homogeneous in the grading of Proposition~\ref{P:KingZ2gr}, consisting entirely of even elements. Thus, we may refer to it and its cyclic permutations as \emph{even triples}.

A record of the admissible triples (in, say, a multiplication table) is a standard way of specifying the octonion algebra $\mathbb O$ \cite[\S2.1]{Baez02}, \cite[\S A.1]{GHW}. This record is best presented geometrically on the basis of the \emph{Fano plane}, the projective plane $\mathsf{PG}_2(2)$ over the field $\mathsf{GF}(2)$ of two elements. Here, the lines are the underlying sets of the admissible triples, and the points are the elements of $\mathbb O$ that appear in the admissible triples. Beyond the Fano plane structure, an orientation to specify the cyclic order on the lines is required \cite[Fig.~A1]{GHW}. The diagram
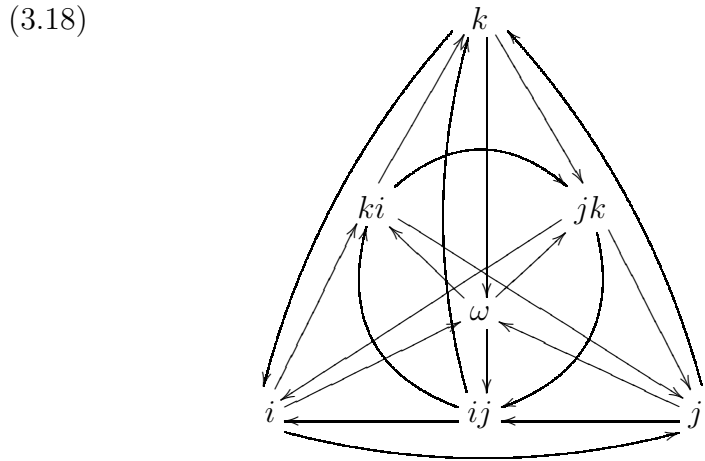
\begin{equation}\label{E:FanoDisp}
\xymatrix{
&
&
k
\ar@<.5ex>[ddd]
\ar@<.5ex>[ddr]
\ar@<-1ex>@/_4mm/[ddddll]
\\
\\
&
ki
\ar@<.5ex>[uur]
\ar@/^8mm/[rr]
\ar@<.5ex>[ddrrr]
&
&
jk
\ar@<.5ex>[ddr]
\ar@/^8mm/[ddl]
\ar@<-.5ex>[ddlll]
\\
&
&
\omega
\ar[ur]
\ar[ul]
\ar@<.5ex>[d]
\\
i
\ar@<.5ex>[uur]
\ar[urr]
\ar@<-1ex>@/_4mm/[rrrr]
&
&
ij
\ar@<.5ex>[ll]
\ar@/^8mm/[uul]
\ar@<.5ex>@/^4mm/[uuuu]
&
&
j
\ar[ull]
\ar@<.5ex>[ll]
\ar@<-1ex>@/_4mm/[uuuull]
}
\end{equation}
\vskip 8mm
\noindent
displays the oriented Fano plane for the admissible triples \eqref{E:admstrpl}.

While the actual display \eqref{E:FanoDisp} itself has fewer symmetries than the oriented Fano plane, the symmetries of the display do line up well with the Kingdon algebra realization of the octonion algebra $\mathbb O$. If \eqref{E:FanoDisp} is ``filled in'' to a two-simplex $\Delta_2$ as the convex hull of $\set{i,j,k}$, then we have the even elements $ij,jk,ki$ as the barycenters of the edges, and the volume element $\omega$ as the barycenter of the simplex. In particular, the display automorphisms preserve the set of even triples.

The symmetries of the display are obtained from automorphisms of $\mathbb O$ that are induced from permutations of $\set{i,j,k}$ according to the universality property of Proposition~\ref{P:UnvPrKVB}. The transposition $(i\ j)$, for instance, reflects the simplex $\Delta_2$ about the vertical $(k,\omega,ij)$ in \eqref{E:FanoDisp}, reversing the orientations along the bottom edge, and also around the projective line $(ij\ ki\ jk)$ of even elements.

\begin{remark}
It is easy to remember the orientations appearing in \eqref{E:FanoDisp}. Those on the boundary of the simplex $\Delta_2$ are determined by the way that admissible triples encode products. The arrows from the vertices are directed in to the volume element, thence out to the edge barycenters. Finally, the orientation of the ``inner circle,'' the projective line containing the edge barycenters, is the opposite of the induced action $(ij\ jk\ ki)$ of the cyclic permutation $(i\ j\ k)$.
\end{remark}

\section{Kingdon algebras on $3$-dimensional formed spaces}\label{S:Real}

\subsection{Alternative exterior algebras}\label{S:altextalg}

The actual dimension $2^n$ of a Clifford algebra over a formed space of dimension $n$ is not immediately apparent from its universality property. The dimension $2^n$ is fixed when the Clifford algebra is taken as a quantization of an exterior algebra \cite[\S2.1]{CC1954}, \cite[p.102]{Chevalley2}. Here, we introduce an eight-dimensional alternative, nonassociative version $\Lambda'(\mathbb R^3)$ of the exterior algebra of the three-dimensional real vector space $\mathbb R^3$, later recognizing it as the Kingdon algebra $K(\mathbb R^3,0)$ for the zero bilinear form on $\mathbb R^3$. Ultimately (\S\ref{SS:QuAlExAl}), we may then interpret any Kingdon algebra on the $3$-dimensional formed space $\mathbb R^3$ as a quantization of $\Lambda'(\mathbb R^3)$.

\begin{definition}\label{D:altextalg}
Consider the $4$-dimensional exterior algebra $\Lambda(\mathbb R^2)$, equipped with the anti-automorphism  $*\colon\Lambda(\mathbb R^2)\to\Lambda(\mathbb R^2)$ that negates $\mathbb R^2=\Lambda(\mathbb R^2)_1$ and fixes $\Lambda(\mathbb R^2)_0\oplus\Lambda(\mathbb R^2)_2$. Then the \emph{alternative exterior algebra} $\Lambda'(\mathbb{R}^3)$ is the result $\Lambda(\mathbb R^2)^0$ of applying the Cayley-Dickson process \eqref{E:GenCa-Di} to $\Lambda(\mathbb R^2)$, with $\gamma=0$.
\end{definition}

\begin{remark}
Following our use of the matronymic ``Kingdon'' for alternative Clifford algebras, we are inclined to describe alternative exterior (Grassmann) algebras as \emph{Medenwald algebras}.
\end{remark}

\begin{lemma}\label{L:altextalg}
Consider the alternative exterior algebra $\Lambda'(\mathbb{R}^3)$.
\begin{enumerate}
\item[$(\mathrm a)$]
It is alternative, but not associative, as the result of applying the Cayley-Dickson process to an algebra which is associative, but not commutative.
\item[$(\mathrm b)$]
It is graded as $\Lambda'(\mathbb{R}^3)=\bigoplus_{h=0}^3(\Lambda'(\mathbb{R}^3))_h$, where the $0$-, $1$-, $2$- and $3$-components have respective bases $\set 1$, $\set{i,j,k}$, $\set{ij,jk,ki}$ and $\set\omega$.
\item[$(\mathrm c)$]
The multiplication structure of $\Lambda'(\mathbb{R}^3)$ is given by Table~\ref{Tb:altextalg}, taking \eqref{E:GenACijk} as the basis for $\Lambda'(\mathbb{R}^3)_1=\mathbb R^3$. The basis element $\omega=(ij)k$ of $\Lambda'(\mathbb{R}^3)_3$ is described as the \emph{volume element} (compare Definition~\ref{D:FormsOfO}).
\end{enumerate} 
\end{lemma}

\renewcommand{\arraystretch}{1.5}
\begin{table}[hbt]\label{Table1}
\centering
\begin{tabular}{|c||c|c|c||c|c|c||c|}
\hline
1 & $i$ & $j$ & $k$ & $ij$ & $jk$ & $ki$ & $\omega$\\
\hline
\hline
$i$ & $0$ & $ij$ & $-ki$ & $0$ & $-\omega$ & $0$ & $0$\\
\hline
$j$ & $-ij$ & $0$ & $jk$ & $0$ & $0$ & $-\omega$ & $0$\\
\hline
$k$ & $ki$ & $-jk$ & $0$ & $-\omega$ & $0$ & $0$ & $0$\\
\hline
\hline
$ij$ & $0$ & $0$ & $\omega$ & $0$ & $0$ & $0$ & $0$\\
\hline
$jk$ & $\omega$ & $0$ & $0$ & $0$ & $0$ & $0$ & $0$\\
\hline
$ki$ & $0$ & $\omega$ & $0$ & $0$ & $0$ & $0$ & $0$\\
\hline
\hline
$\omega$ & $0$ & $0$ & $0$ & $0$ & $0$ & $0$ & $0$\\
\hline
\end{tabular}
\vskip 3mm
\caption{The multiplication table of the $8$-dimensional alternative exterior algebra.}\label{Tb:altextalg}
\end{table}
\renewcommand{\arraystretch}{1}

Definition~\ref{D:altextalg} may be extended to more general fields.

\begin{definition}\label{D:AltExAlg}
Suppose that $F$ is a field in which $2$ is invertible. Then the \emph{alternative exterior algebra} $\Lambda'(F^3)$ is the eight-dimensional algebra with basis
\begin{equation}\label{E:BasLmbda}
\set{1,i,j,k,ij,jk,ki,\omega}
\end{equation}
and multiplication given by Table~\ref{Tb:altextalg}.
\end{definition}

\begin{remark}
(a)
By  Lemma~\ref{L:altextalg}(c), Definition~\ref{D:AltExAlg} is consistent with Definition~\ref{D:altextalg} in the case where $F=\mathbb R$.
\vskip 2mm
\noindent
(b)
By Lemma~\ref{L:altextalg}(a), the ring $\Lambda'(\mathbb{R}^3)$ is alternative. Therefore, its subring $\Lambda'(\mathbb{Z}^3)$ generated by \eqref{E:BasLmbda} is also alternative. It follows that for each field $F$ where $2$ is invertible, the alternative exterior algebra $\Lambda'(F^3)=F\otimes\Lambda'(\mathbb{Z}^3)$ is indeed alternative.
\vskip 2mm
\noindent
(c)
As noted in \cite[Prop.~3.1]{QAA} in the context of an algebraically closed field $F$, the linear span of $\set{i,j,k,ij,jk,ki,\omega}$ in $\Lambda'(F^3)$ is a seven-dimensional Maltsev algebra: it is closed under multiplication, and each side of the Maltsev identity
$$
(xy)(xz)=((xy)z)x+((yz)x)x+((zx)x)y
$$
reduces to $0$ on this subspace. Compare \cite[Th.~4.1(i)]{QAA}.
\vskip 2mm
\noindent
(d)
The distinction between the associative exterior algebra $\Lambda(F^3)$ and the alternative exterior algebra $\Lambda'(F^3)$ may be appreciated by noting the equation $(i\wedge j)\wedge k=-k\wedge(j\wedge i)$ which holds in the former, in contrast with the palindromic relation $(ij)k=k(ji)$ that holds in the latter.
\end{remark}

\begin{proposition}\label{P:LamPrKIin}
Let $F$ be a field where $2$ is invertible. Then $\Lambda'(F^3)$ is (isomorphic to) the Kingdon algebra $K(F^3,0)$.
\end{proposition}

\begin{proof}
Consider the instance 
$$
\xymatrix{
K(F^3,0)
\ar@{-->}[dr]^{\Theta}
\\
\,F^3\,
\ar[u]^{\iota}
\ar@{^{(}->}[r]_-{\theta}
&
\Lambda'(F^3)
}
$$
of the universality property for Kingdon algebras (Proposition~\ref{P:UnvPrKVB}), which applies by Lemma~\ref{L:standardgenerators} thanks to the anti-symmetry of the lower right $(7\times 7)$-subtable of Table~\ref{Tb:altextalg}.

Because $F^3$ generates the algebra $\Lambda'(F^3)$, the homomorphism $\Theta$ is surjective, whence $\dim K(F^3,0)\ge 8$. On the other hand, Theorem~\ref{T:reductionstep} and the relations of Table~\ref{Tb:altextalg} show that the basis set \eqref{E:BasLmbda} offers a full set of representatives for the cosets comprising $K(F^3,0)=\mathrm{Alt}[F^3]$, so $\dim K(F^3,0)\le 8$. Thus $\dim K(F^3,0)=8$, and $\Theta$ is the required isomorphism.
\end{proof}

We now use Kingdon algebras to give an appropriate definition of the alternative exterior algebra over a vector space of arbitrary dimension, assuming that the base field is not of characteristic 2.

\begin{definition}\label{D:gnltexal}
Suppose that $2$ is invertible in a field $F$. Then the \emph{alternative exterior algebra} $\Lambda'(V)$ over an $F$-vector space $V$ is defined as $K(V,0)$.
\end{definition}

\begin{problem}\label{P:gnltexal}
Investigate the structure and geometric significance of the general alternative exterior algebras of Definition~\ref{D:gnltexal}.
\end{problem}

\subsection{Kingdon algebras and Cayley-Dickson algebras}

Let $F$ be a field where $2$ is invertible. We have the following generalization of Theorem~\ref{T:constructingoct}.

\begin{theorem}\label{T:cliffordtocayley}
A Kingdon algebra $K(V,B)$ over a formed space $(V,B)$ of dimension $3$ is isomorphic to an eight-dimensional Cayley-Dickson algebra.
\end{theorem}

\begin{proof}
Suppose that the basis $\set{i,j,k}$ of $V$ is orthogonal with respect to the symmetric bilinear form $B$ on $V$, where
\begin{align}\label{E:BijkVEps} 
B(i,i)=2\varepsilon_1\,,\
B(j,j)=2\varepsilon_2\,,\
B(k,k)=2\varepsilon_3\,.
\end{align}
We will be considering the Cayley-Dickson algebra $ F^{\varepsilon_1\varepsilon_2\varepsilon_3}$, where the successive extensions
\begin{equation}\label{E:CDepslon}
\begin{cases}
&F^{\varepsilon_1}
=F[i]
\
\mbox{with}
\
i^2=\varepsilon_1\,,
\\
&F^{\varepsilon_1\varepsilon_2}
=F^{\varepsilon_1}[j]
\
\mbox{with}
\
j^2=\varepsilon_2\,, \mbox{and}
\\
&F^{\varepsilon_1\varepsilon_2\varepsilon_3}
=F^{\varepsilon_1\varepsilon_2}[k]
\
\mbox{with}
\
k^2=\varepsilon_3
\end{cases}
\end{equation}
witness the construction process.
Lemma~\ref{L:SisClosd} shows that $K(V,B)$ is an at most $8$-dimensional alternative algebra over the field $F$.

We have an injective linear map $\psi:V\to F^{\varepsilon_1\varepsilon_2\varepsilon_3}$ which takes $i$ to $i$, $j$ to $j$ and $k$ to $k$. Then, by Corollary~\ref{L:UPFA}, we have the commutative diagram 
$$
\xymatrix{
\mathrm{Alt}[V]
\ar@{-->}[dr]^{\Psi}
\\
V
\ar[r]_-{\psi}
\ar@{^(->}[u]
&
F^{\varepsilon_1\varepsilon_2\varepsilon_3}
}
$$
with an algebra homomorphism $\Psi$ extending $\psi$. Since the image of $V$ generates all of $F^{\varepsilon_1\varepsilon_2\varepsilon_3}$, the homomorphism $\Psi$ is surjective. By \eqref{E:plindrCD}, \eqref{E:CDepslon} and the orthogonality of $\set{i,j,k}$, the kernel of $\Psi$ includes the ideal $J_B$ of $\mathrm{Alt}[V]$ generated by
$$
\set{ab+ba-B(a,b)\cdot 1, (ab)c-c(ba)|a,b,c\in\set{i,j,k}}
$$ 
(cf. Lemma~\ref{L:standardgenerators}). Thus, we have a surjective algebra homomorphism
\begin{equation}\label{E:PsiPrime}
\Psi':K(V,B)=\mathrm{Alt}[V]/J_B\to F^{\varepsilon_1\varepsilon_2\varepsilon_3}\,.
\end{equation} 
This map is a surjective linear transformation from a space of dimension at most $8$ to one of dimension exactly $8$. Therefore, $K(V,B)$ must also have dimension at least $8$, and hence exactly $8$. Thus \eqref{E:PsiPrime} is an isomorphism.
\end{proof}

\begin{corollary}\label{P:quadratic}
The Kingdon algebra $K(V,B)$ over a $3$-dimensional formed space $(V,B)$ is symmetric quadratic.
\end{corollary}

\begin{proof}
Cayley-Dickson algebras are symmetric quadratic.
\end{proof}

\subsection{Norms, forms, and simplicity}\label{SS:KingNorm}

\begin{theorem}\label{T:KingNorm}
In the context of Theorem~\ref{T:cliffordtocayley}, the formula
\begin{align}\label{E:KingNorm}
N(a_0)
&
-\varepsilon_1N(a_1)-\varepsilon_2N(a_2)-\varepsilon_3N(a_3)
\\ \notag
&
+\varepsilon_1\varepsilon_2N(a_4)
+\varepsilon_2\varepsilon_3N(a_5)
+\varepsilon_3\varepsilon_1N(a_6)
-\varepsilon_1\varepsilon_2\varepsilon_3N(a_7)
\end{align}
gives the norm $N(x)$ of an element
\begin{equation}\label{E:KngAlgEl}
x=a_0+a_1i+a_2j+a_3k+a_4ij+a_5jk+a_6ki+a_7\omega
\end{equation}
of $K(V,B)$.
\end{theorem}

\begin{proof}
The proof involves three applications of the second relation of \eqref{E:TrNrGamm}, once for each of the Cayley-Dickson extensions recorded in \eqref{E:CDepslon}. The formula
\begin{equation}\label{E:FirstNrm}
N(a_0+a_1i)=N(a_0)-\varepsilon_1N(a_1)
\end{equation}
is obtained at the first step. For the second step, we have
$$
a_0+a_1i+a_2j+a_4ij=(a_0+a_1i)+(a_2+a_4i)j
$$
and $N(a_0+a_1i+a_2j+a_4ij)=N\big((a_0+a_1i)+(a_2+a_4i)j\big)\overset{\eqref{E:TrNrGamm}}=$
\begin{align}\notag
&
N(a_0+a_1i)-\varepsilon_2N(a_2+a_4i)
\\ \notag
&
\overset{\eqref{E:FirstNrm}}
=
N(a_0)-\varepsilon_1N(a_1)-\varepsilon_2\big[N(a_2)-\varepsilon_1N(a_4)\big]
\\ \label{E:SecndNrm}
&
=
N(a_0)-\varepsilon_1N(a_1)-\varepsilon_2N(a_2)+\varepsilon_1\varepsilon_2N(a_4)\,.
\end{align}
At the third step, recalling $\omega=(ij)k$ and $ki=-ik$, we have $x\overset{\eqref{E:KngAlgEl}}=$
\begin{align*}
a_0&+a_1i+a_2j+a_3k+a_4ij+a_5jk+a_6ki+a_7(ij)k
\\
&
=(a_0+a_1i+a_2j+a_4ij)+(a_3k+a_5j-a_6i+a_7ij)k
\end{align*}
and $N(x)=N\big((a_0+a_1i+a_2j+a_4ij)+(a_3-a_6i+a_5j+a_7ij)k\big)$
\begin{align*}
\overset{\eqref{E:TrNrGamm}}
=
&
N(a_0+a_1i+a_2j+a_4ij)-\varepsilon_3N(a_3-a_6i+a_5j+a_7ij)
\\
\overset{\eqref{E:SecndNrm}}
=
&
N(a_0)-\varepsilon_1N(a_1)-\varepsilon_2N(a_2)+\varepsilon_1\varepsilon_2N(a_4)
\\
&
\rule{6mm}{0mm}
-\varepsilon_3
\big[N(a_3)-\varepsilon_1N(-a_6)-\varepsilon_2N(a_5)+\varepsilon_1\varepsilon_2N(a_7)\big]
\\
=
&
N(a_0)-\varepsilon_1N(a_1)-\varepsilon_2N(a_2)+\varepsilon_1\varepsilon_2N(a_4)
\\
&
\rule{6mm}{0mm}
-\varepsilon_3N(a_3)+\varepsilon_3\varepsilon_1N(-a_6)+\varepsilon_2\varepsilon_3N(a_5)-\varepsilon_1\varepsilon_2\varepsilon_3N(a_7)\,,
\end{align*}
as required.
\end{proof}

\begin{corollary}\label{C:KingNorm}\phantom{pog}\\
$(\mathrm a)$
The bilinear form $B$ on the formed space $(V,B)$ is degenerate if and only if the norm $N$ on the quadratic algebra $K(V,B)$ is degenerate.
\vskip 2mm
\noindent
$(\mathrm b)$
The $8$-dimensional algebra $K(V,B)$ is simple if and only if the bilinear form $B$ is nondegenerate.
\end{corollary}

\begin{proof}
The proof of (a) is a direct consequence of the theorem. Then, if a flexible quadratic $F$-algebra is simple, its norm is nondegenerate \cite[Prop.~2.1(iv)]{QAA}. Conversely, if the norm of a quadratic $F$-algebra is nondegenerate, then the algebra is either simple or isomorphic to the algebra $F^2$ of Example~\ref{X:LorenzQA} \cite[Prop.~2.1(iv)]{QAA}. Of course, the latter option does not involve algebras of dimension $8$, and (b) follows since the alternative algebra $K(V,B)$ is flexible (Corollary~\ref{C:AltFlexi}).
\end{proof}

\subsection{Quantizing the alternative exterior algebra}\label{SS:QuAlExAl}

In Table~\ref{Tb:Mlt8-dim}, we present the multiplication table for an $8$-dimensional Kingdon algebra $K(V,B)$ in terms of its standard basis $\set{1,i,j,k,ij,jk,ki,\omega=(ij)k}$.

\renewcommand{\arraystretch}{1.5}
\begin{table}[hbt]
\scalebox{0.8}
{
\begin{tabular}{|c|ccc|ccc|c|}
\hline
1 & $i$ & $j$ & $k$ & $ij$ & $jk$ & $ki$ & $\omega$ \\
\hline
$i$ & $Q(i)$ & $ij$ & $-ki$ & $Q(i)j$ & $-\omega$ & $-Q(i)k$ & $-Q(i)jk$\\
$j$ & $-ij$ & $Q(j)$ & $jk$ & $-Q(j)i$ & $Q(j)k$ & $-\omega$ & $-Q(j)ki$\\
$k$ & $ki$ & $-jk$ & $Q(k)$ & $-\omega$ & $-Q(k)j$ & $Q(k)i$ & $-Q(k)ij$\\
\hline
$ij$ & $-Q(i)j$ & $Q(j)i$ & $\omega$ & $-Q(i)Q(j)$ & $Q(j)ki$ & $-Q(i)jk$ & $-Q(i)Q(j)k$\\ 
$jk$ & $\omega$ & $-Q(j)k$ & $Q(k)j$ & $-Q(j)ki$ & $-Q(j)Q(k)$ & $Q(k)ij$ & $-Q(j)Q(k)i$\\
$ki$ & $Q(i)k$ & $\omega$ & $-Q(k)i$ & $Q(i)jk$ & $-Q(k)ij$ & $-Q(k)Q(i)$ & $-Q(k)Q(i)j$\\
\hline
$\omega$ & $Q(i)jk$ & $Q(j)ki$ & $Q(k)ij$ & $Q(i)Q(j)k$ & $Q(j)Q(k)i$ & $Q(k)Q(i)j$ & $Q(i)Q(j)Q(k)$\\
\hline
\end{tabular}
}
\vskip 3mm
\caption{The multiplication table of an eight-dimensional Kingdon algebra.}\label{Tb:Mlt8-dim}
\end{table}
\renewcommand{\arraystretch}{1}

Comparison with Table~\ref{Tb:altextalg} affords a different perspective on an eight-dimensional Kingdon algebra $K(V,B)$, namely as a \emph{deformation} or \emph{quantization} of the alternative exterior algebra $\Lambda'(\mathbb R^3)$ by introduction of the parameters $Q(i),Q(j),Q(k)$ along the three chosen perpendicular axes of the space $\mathbb R^3$.

\section{Real Kingdon algebras}

\subsection{Orthogonal bases}

We seek to classify all real Kingdon algebras on $3$-dimensional formed spaces, and examine various aspects of their structure. Each bilinear form $B$ on a three-dimensional real vector space $V$ has an orthogonal basis of elements that square to $0$ (i.e., are \emph{isotropic}) or that square to $\pm 1$. This allows us to give an alternative labelling for real $8$-dimensional Kingdon algebras, which may also be applied to real Kingdon algebras of other dimensions.

\begin{definition}\label{D:Kb0b1b-1}
Consider an orthogonal basis for a real formed space $(V,B)$ where each basis vector $e$ satisfies $B(e,e)\in\set{0,\pm2}$. Set $K(b_0,b_1,b_{-1})=K(V,B)$, where for each ``length'' $l\in\set{0,\pm1}$, the number of basis elements with $B(e,e)=2l$ is $b_l$.
\end{definition}

The notation of Definition~\ref{D:Kb0b1b-1} recalls the Clifford algebra notation of \eqref{E:Cb0b1b-1}. Indeed, $K(b_0,b_1,b_{-1})=\mathrm{Cl}[b_0,b_1,b_{-1}]$ if $b_0+b_1+b_{-1}<3$ (\S\ref{SS:LowDimEx}).

\subsection{Centers, nuclei, and commutants}\label{SS:CeNucCom}

\begin{definition}\label{D:ComNuCen}
Let $(M,\cdot)$ be a magma, where a product $x\cdot y$ may also be denoted by the simple juxtaposition $xy$ that binds more strongly than the explicit multiplication. Thus $xy\cdot z=(x\cdot y)\cdot z$, for example. 
\begin{enumerate}
\item[$(\mathrm a)$]
The subset
$$
C=\set{c|\forall\ x\in M,\ x\cdot c=c\cdot x}
$$
of $M$ is the \emph{commutant} of $(M,\cdot)$.
\item[$(\mathrm b)$]
The subset $N=$
$$
\set{n|\forall\ x,y\in M,\ n\cdot xy=nx\cdot y\,,\ x\cdot ny=xn\cdot y\,,\ x\cdot yn=xy\cdot n}
$$
of $M$ is the \emph{nucleus} of $(M,\cdot)$.
\item[$(\mathrm c)$]
The \emph{center} of $(M,\cdot)$ is the intersection $Z=C\cap N$.
\end{enumerate}
\end{definition}

For a three-dimensional formed space $(V,B)$ spanned by the set $\set{i,j,k}$, whose elements are described as \emph{generators}, we will take 
\begin{equation}\label{E:GenElemx}
x=a_0+a_1 i+a_2 j+a_3 k+a_4 ij+a_5 jk+ a_6 ki+a_7\omega
\end{equation} 
as a general element of $K(V,B)$, written in terms of the standard basis used in \S\ref{SS:QuAlExAl}.

\subsubsection{Simple algebras}

By Theorem~\ref{T:constructingoct} and Corollary~\ref{C:KingNorm}(b), the (split) octonions are simple. Their centers, nuclei and commutants all coincide with the field $\mathbb{R}$. Indeed, for any non-real element $x_1$ of the algebra, one may find further elements $x_2,x_3$ such that $x_1,x_2,x_3$ are mutually $\mathbb R$-algebraically independent. Then $x_1$ does not lie in the commutant, since it does not commute with $x_2$. Furthermore, it does not lie in the nucleus or the center, since the associator $(x_1,x_2,x_3)$ is nonzero.

\subsubsection{Commutants of non-simple algebras}

In the remaining cases, the situation is more complicated.

\begin{lemma}\label{L:NoSimCom}
Let $K(V,B)$ be a non-simple Kingdon algebra.
\begin{enumerate}
\item[$(\mathrm a)$]
If $B=0$, so that $K(V,B)=\Lambda'(\mathbb{R}^3)$, then the commutant is the span of $\set{1,\omega}$.
\item[$(\mathrm b)$]
If $B\ne0$, then the commutant is $\mathbb R$.
\end{enumerate}
\end{lemma}

\begin{proof}
(a)
In the $\Lambda'(\mathbb{R}^3)$ case, $\omega x=x \omega$ for all elements $x$ of $K(V,B)$, so $\omega$ is in the commutant. 
Now suppose that in \eqref{E:GenElemx}, a coefficient $a_l$ for $0<l<7$ is nonzero. Suppose that $a_lv$ is the corresponding summand in \eqref{E:GenElemx}. Then $v$ will anticommute with at least one of the vectorial basis elements $i$, $j$ and $k$: with exactly one if $v$ is the product of two vectorial basis elements, and exactly two if it is itself a vectorial element. Since $v$ does not commute with a basis element, neither does $x$.
\vskip 2mm
\noindent
(b)
Suppose $B\ne0$. Suppose that in \eqref{E:GenElemx}, we have $a_l\ne0$ for some $0<l\le7$, where $a_l$ is the coefficient of basis element $v$. Then $x$ does not commute with some vectorial basis element of $K(V,B)$ for the same reason as in (a), save in the case of $l=7$. When $a_7\ne0$, there exists a standard basis element $w\neq 1$ such that $w\omega\ne0$, and in that case $w\omega=-\omega w$, so $x$ is not in the commutant.
\end{proof}

\subsubsection{Nuclei and centers}

\begin{lemma}
The nucleus of the alternative exterior algebra $\Lambda'(\mathbb R^3)$ has basis $\set{1,ij,jk,ki,\omega}$.
\end{lemma}
\begin{proof}
From Table~\ref{Tb:altextalg}, it is clear that $\omega$ is nuclear. If $a\in \set{ij,jk,ki}$ and $y,z\in\set{i,j,k,ij,jk,ki,\omega}$, then 
\begin{equation}\label{E:aNucleus}
a(yz)=(ay)z=0\,,\
y(az)=(ya)z=0\,,\mbox{and }
(yz)a=y(za)=0\,, 
\end{equation}
since there are at least four copies of generators in each product, and only three unique generators. Linear extension of \eqref{E:aNucleus} shows that the span of $\set{ij,jk,ki}\cup\set{1,\omega}$ is a subset of the nucleus.

Now consider a general element \eqref{E:GenElemx} of $\Lambda'(\mathbb R^3)$. If the coefficient $a_1$ is nonzero, we have $x(jk)=a_0jk-a_1\omega$, while $(xj)k=a_0jk+a_1\omega$. Similar reasoning holds if $a_2$ or $a_3$ is nonzero.
\end{proof}

\begin{corollary}
The center of the alternative exterior algebra $\Lambda'(\mathbb R^3)$ has basis $\set{1,\omega}$.
\end{corollary}

\begin{proof}
By Lemma~\ref{L:NoSimCom}(a), the commutant is spanned by $\set{1,\omega}$.
\end{proof}

\begin{lemma}
The algebras $K(1,a_1,a_{-1})$ have $\mathbb{R}$ as their nuclei.
\end{lemma}

\begin{proof}
Suppose that $i$ is the unique $B$-orthogonal basis element such that $i^2=0$, while $j$ and $k$ are the remaining two orthogonal basis elements. Consider a general element \eqref{E:GenElemx}.
\begin{itemize}
\item
If any of $a_1,a_4,a_5,a_7$ are nonzero, then
\begin{align*}
x(jk)&=a_0jk-a_1\omega+a_2Q(j)k-a_3Q(k)j+a_4Q(j)ki
\\
&
\rule{30mm}{0mm}
-a_5Q(k)ij-a_6Q(j)Q(k)+a_7Q(j)Q(k)i\,,
\end{align*}
while
\begin{align*}
(xj)k&=a_0jk+a_1\omega+a_2Q(j)k-a_3Q(k)j-a_4Q(j)ki
\\
&
\rule{30mm}{0mm}
+a_5Q(k)ij-a_6Q(j)Q(k)-a_7Q(j)Q(k)i\,,
\end{align*}
and $x$ is not in the nucleus.
\item
If $a_2\neq 0$, then $x(ik)=-a_0ki+a_2\omega-a_3Q(k)i-a_5Q(k)ij$, while $(xi)k=-a_0ki-a_2\omega-a_3Q(k)i +a_5Q(k)ij$. Thus $x$ is not in the nucleus. 
\item
Similar reasoning holds for $a_3\neq 0$, except we use $i$ and $j$ instead of $i$ and $k$. 
\item
If $a_6\neq 0$, then $((jk)k)x=Q(k)jx=$
\begin{align*}
Q(k)\big[a_0j&-a_1ij+a_2Q(j)+a_3jk
\\
&
-a_4Q(j)i+a_5Q(j)k-a_6\omega-a_7Q(j)ki\big]\,, 
\end{align*}
while $(jk)(kx)=$
\begin{align*}
Q(k)\big[a_0&j-a_1ij+a_2Q(j)+a_3jk
\\
&
+a_4Q(j)i-a_5Q(j)k+a_6\omega+a_7Q(j)ki\big]\,.
\end{align*} 
Since $a_6\neq -a_6$, and both $Q(j)$ and $Q(k)$ are nonzero, then $x$ is not in the nucleus. 
\end{itemize}
Thus $x$ is only in the nucleus if $a_1=a_2=a_3=a_4=a_5=a_6=a_7=0$, i.e., if $x\in\mathbb{R}$.
\end{proof}

\begin{corollary}
The algebras $K(1,a_1,a_{-1})$ have $\mathbb{R}$ as their centers.
\end{corollary}

\begin{proof}
By Lemma~\ref{L:NoSimCom}(b), the commutant is $\mathbb{R}$.
\end{proof}

\begin{lemma}
The algebras $K(2,a_1,a_{-1})$, where $i^2=j^2=0$, have nuclei with basis $\set{1,ij,\omega}$.
\end{lemma}

\begin{proof}
The basis elements $ij$ and $\omega$ associate with each pair $a,b$ of elements of $\set{i,j,k,ij,jk,ik,\omega}$. If one of $a$ and $b$ contains a copy of $i$ or $j$, we have 
$$
(ab)(ij)=a(b(ij)=0\,, (a(ij))b=a((ij)b)=0\,, (ij)(ab)=((ij)a)b=0\,,
$$
and similarly for $\omega$ in place of $ij$. Therefore, the only interesting case is when we have distinct $a$ and $b$ chosen from among the basis elements not containing a copy of $i$ or $j$. But this is impossible, since the only such basis element is $k$. Thus, $\mathbb{R}+ij\mathbb{R}+\omega\mathbb{R}$ is contained in the nucleus.

Now, consider a general element \eqref{E:GenElemx}.
\begin{itemize}
\item
If $a_1\neq 0$ or $a_6\neq 0$, then 
\begin{align*}
(jk)x
&
=a_0jk+a_1\omega+a_3Q(k)j-a_6Q(k)ij\,,\mbox{ while}
\\
j(kx)
&
=a_0jk-a_1\omega+a_3Q(k)j+a_6Q(k)ij\,,
\end{align*}
so $x$ is not in the nucleus.
\item
Similarly, if $a_2\neq 0$ or $a_5\neq 0$, then
\begin{align*}
(ik)x
&
=-a_0ki-a_2\omega+a_3Q(k)i+a_5Q(k)ij\,,\mbox{ while}
\\
i(kx)
&
=-a_0ki+a_2\omega+a_3Q(k)i-a_5Q(k)ij\,,
\end{align*}
so $x$ is not in the nucleus.
\item
If $a_3\neq 0$, then $(ij)x=a_0ij+a_3\omega$, while $i(jx)=a_0ij-a_3\omega$, so $x$ is not in the nucleus.
\end{itemize}
Thus $x$ is only in the nucleus if $a_1=a_2=a_3=a_5=a_6=0$, proving the desired result.
\end{proof}

\begin{corollary}
The algebras $K(2,a_1,a_{-1})$ have $\mathbb{R}$ as their center.
\end{corollary}

\begin{proof}
By Lemma~\ref{L:NoSimCom}(b), the commutant is $\mathbb{R}$.
\end{proof}

\subsection{The full classification}\label{SS:FC}

Consideration of the nuclei shows that there is no isomorphism between algebras of the form $K(2,b_1,b_{-1})$ and those of the form $K(1,b_1,b_{-1})$, or between algebras of either of these forms and the alternative exterior algebra. We now show when we do have isomorphisms within each of these classes.

\begin{lemma}
$K(2,1,0)\not\cong K(2,0,1)$
\end{lemma}

\begin{proof}
While the first of these contains a copy $\mathrm{Cl}(0,1,0)$ of the split complex numbers, the second does not.
\end{proof}

\begin{lemma}
$K(1,1,1)\cong K(1,2,0)$
\end{lemma}

\begin{proof}
We have a subalgebra isomorphic to the split quaternions in both cases, and the third generator of each algebra squares to $0$.

Suppose $v_1$, $v_2$ and $v_3$ are the generators of the first algebra, and $w_1$, $w_2$ and $w_3$ the generators of the second algebra, with $v_1^2=w_1^2=0$, $v_2^2=w_2^2=w_3^2=1$ and $v_3^2=-1$. Then the vector space isomorphism 
$$
\scalebox{0.84}
{
\xymatrix{
1
\ar@{|->}[d]
&
v_1
\ar@{|->}[d]
&
v_2
\ar@{|->}[d]
&
v_3
\ar@{|->}[d]
&
v_1v_2
\ar@{|->}[d]
&
v_1v_3
\ar@{|->}[d]
&
v_2v_3
\ar@{|->}[d]
&
(v_1v_2)v_3
\ar@{|->}[d]
\\
1
&
w_1
&
w_2
&
w_2w_3
&
w_1w_2
&
-(w_1w_2)w_3
&
w_3
&
-w_1w_3
}
}
$$
is an algebra isomorphism.
\end{proof}

\begin{lemma}
$K(1,1,1)\not\cong K(1,0,2)$.
\end{lemma}

\begin{proof}
While the first of these algebras contains a copy $\mathrm{Cl}(0,1,1)$ of the split quaternions, the second does not.
\end{proof}

Table~\ref{Tb:ClassTab} presents a full classification of all the eight-dimensional real Kingdon algebras. The centers and nuclei are described in terms of the standard basis $\set{1,i,j,k,ij,jk,ik,\omega}$. The table also identifies algebras to which the given Kingdon algebra is isomorphic, where applicable.

\renewcommand{\arraystretch}{1.5}
\begin{table}[hbt]
\centering
\small
\begin{tabular}{c||c|c|c|c|c|}
Algebra & Simple? & $\begin{matrix}\mbox{Center}\\ \mbox{basis}\end{matrix}$& $\begin{matrix}\mbox{Nucleus}\\ \mbox{basis}\end{matrix}$& $\begin{matrix}\mbox{Division}\\ \mbox{algebra?}\end{matrix}$ & $\cong$\\
\hline
\hline
$K(3,0,0)$ & No & $\set{1,\omega}$ & $\set{1,ij,jk,ik,\omega}$ & No & $\Lambda'(\mathbb{R}^3)$\\
\hline
$K(2,0,1)$ & No & $\set1$ & $\set{1,ij,\omega}$ ($k^2=-1$) & No & -\\
\hline
$K(2,1,0)$ & No & $\set1$ & $\set{1,ij,\omega}$ ($k^2=1$) & No & -\\
\hline
$K(1,1,1)$ & No & $\set1$ & $\set1$  & No & $K(1,2,0)$\\
\hline
$K(1,0,2)$ & No & $\set1$ & $\set1$ & No & -\\
\hline
$K(0,3,0)$ & Yes & $\set1$ & $\set1$ & No & $\widetilde{\mathbb{O}}$\\
\hline
$K(0,2,1)$ & Yes & $\set1$ & $\set1$ & No & $\widetilde{\mathbb{O}}$\\
\hline
$K(0,1,2)$ & Yes & $\set1$ & $\set1$ & No & $\widetilde{\mathbb{O}}$\\
\hline
$K(0,0,3)$ & Yes & $\set1$ & $\set1$ & Yes & $\mathbb{O}$\\
\hline
\end{tabular}
\vskip 3mm
\caption{Real Kingdon algebras of dimension $8$.}\label{Tb:ClassTab}
\end{table}
\renewcommand{\arraystretch}{1}

\subsection{Comparison with the Cayley-Dickson construction}\label{S:cayley}

Each $\mathbb R$-algebra presented in Table~\ref{Tb:ClassTab} may be obtained by the Cayley-Dickson process, according to Theorem~\ref{T:cliffordtocayley}. As in the proof of that theorem \eqref{E:BijkVEps}, suppose that the set $\set{i,j,k}$ is orthogonal with respect to a symmetric bilinear form $B$ on $V$, where
\begin{align*}
B(i,i)=2\varepsilon_1\,,\
B(j,j)=2\varepsilon_2\,,\
B(k,k)=2\varepsilon_3
\end{align*}
and $\varepsilon_h\in\set{0,\pm1}$ for $1\le h\le3$. The Kingdon algebra $K(V,B)$ is then implemented as the Cayley-Dickson algebra $\mathbb R^{\varepsilon_1\varepsilon_2\varepsilon_3}$. The diagram
\begin{equation}\label{E:KingCayl}
\scalebox{0.87}
{
\xymatrix{
&
&
&
&
\\
&
&
\boxed{
\protect{
\begin{matrix}
\Lambda'(\mathbb R^3)=
\\
K(3,0,0)
\end{matrix}
}
}
&
&
\\
&
K(2,1,0)
&
\boxed{
\protect{
\begin{matrix}
\Lambda(\mathbb R^2)=
\\
\mathrm{Cl}[2,0,0]
\end{matrix}
}
}
\ar[u]^{0}
\ar[l]_{1}
\ar[r]^{-1}
&
K(2,0,1)
&
\\
K(1,1,1)
\ar@{=}[d]
&
\mathrm{Cl}[1,1,0]
\ar[u]^{0}
\ar[l]_{-1}
\ar[dl]^{1}
&
\boxed{
\protect{
\begin{matrix}
\Lambda(\mathbb R)=
\\
\mathrm{Cl}[1,0,0]
\end{matrix}
}
}
\ar[r]^-{-1}
\ar[l]_-{1}
\ar[u]^-{0}
&
\mathrm{Cl}[1,0,1]
\ar[dr]^-{-1}
\ar[u]_{0}
\ar`[ur]^-1`[uuu]`[lll][lll]
&
\\
K(1,2,0)
&
\boxed{
\protect{
\begin{matrix}
\mathbb R^2=
\\
\mathrm{Cl}[0,1,0]
\end{matrix}
}
}
\ar[u]^-{0}
\ar@<.7ex>[d]^-{1}
\ar[dr]^-{-1}
&
\mathbb R
\ar[r]^-{-1}
\ar[l]_-{1}
\ar[u]^-{0}
&
\boxed{
\protect{
\begin{matrix}
\mathbb C=
\\
\mathrm{Cl}[0,0,1]
\end{matrix}
}
}
\ar[u]_-{0}
\ar[d]^-{-1}
\ar[dl]_-{1}
&
K(1,0,2)
\\
&
\save [].[r]*[F]\frm{}
\mathrm{Cl}[0,2,0]
\ar[ul]^-{0}
\ar[d]^-{1}
\ar[dl]_-{-1}
&
\mathrm{Cl}[0,1,1]=\mathbb R_2^2
\ar@{=}[l]
\ar[d]_-{-1}
\restore
&
\boxed{
\protect{
\begin{matrix}
\mathbb H=
\\
\mathrm{Cl}[0,0,2]
\end{matrix}
}
}
\ar[ur]_{0}
\ar[d]^-{-1}
\ar[dl]_-{1}
\\
\save [].[r].[rr]*[F]\frm{}
K(0,2,1)
&
K(0,3,0)
\ar@{=}[l]
\ar@{=}[r]
&
K(0,1,2)=\widetilde{\mathbb O}
\restore
&
\boxed{
\protect{
\begin{matrix}
\mathbb O=
\\
K(0,0,3)
\end{matrix}
}
}
}
}
\end{equation}
presents a full record of how each $8$-dimensional Kingdon algebra from Table~\ref{Tb:ClassTab} emerges as a Cayley-Dickson algebra. Lower-dimensional real Kingdon algebras are labelled with the Clifford algebra notation of \eqref{E:Cb0b1b-1}, while the  $8$-dimensional Kingdon algebras are labelled with the notation of Definition~\ref{D:Kb0b1b-1}. Diagram \eqref{E:KingCayl} should be viewed as a directed graph. Vertices at a distance of $d$ from the central vertex $\mathbb R$, as Kingdon algebras $K(W,B)$ for a formed space $(W,B)$ of dimension $d$, represent Kingdon algebras of dimension $2^d$.

The central part of the graph, namely the $2$-neighborhood of $\mathbb R$, incorporates the diagram \eqref{E:ClifCayl} which represented Clifford algebras as Cayley-Dickson algebras. Note that each vertex of this $2$-neighborhood has out-degree $3$, with edges labelled $0$, $1$, or $-1$. The graph vertex $\mathbb R_2^2$ is presented in extended form for clarity, as are the extreme vertices $K(1,1,1)$ and $\widetilde{\mathbb O}$.

\section*{Acknowledgement}

The first author expresses his thanks to Jonas Hartwig for the advice he gave early in this project, and for feedback during early presentations of the ideas in the paper.

\end{document}